\documentclass[10pt,a4paper,twoside,reqno]{amsart}
\usepackage{amssymb, verbatim, amsmath,amsthm,amsfonts}
\usepackage{color}

\usepackage[all]{xy}
\usepackage{amsthm}
\usepackage{amssymb}

\usepackage{MnSymbol}
\usepackage{pifont}
\usepackage{ifsym}
\usepackage{graphicx} 
\usepackage[hidelinks]{hyperref}

\usepackage[a4paper,
            bindingoffset=0.2in,
            left=1in,
            right=1in,
            top=1in,
            bottom=1in,
            footskip=.25in]{geometry}

\newcommand{\bc}{\begin{center}}
\newcommand{\ec}{\end{center}}
\newcommand{\be}{\begin{enumerate}}
\newcommand{\ee}{\end{enumerate}}
\newcommand{\beq}{\begin{equation}}
\newcommand{\eeq}{\end{equation}}
\newcommand{\bi}{\begin{itemize}}
\newcommand{\ei}{\end{itemize}}
\newcommand{\bd}{\begin{description}}
\newcommand{\ed}{\end{description}}
\newcommand{\ba}{\begin{array}}
\newcommand{\bea}{\begin{eqnarray*}}
\newcommand{\eea}{\end{eqnarray*}}
\newcommand{\ea}{\end{array}}
\newcommand{\bt}{\begin{tabular}}
\newcommand{\et}{\end{tabular}}

\newcommand{\fr}[1]{\mathfrak{#1}}

\newcommand{\mb}{\mbox}

\newcommand{\bmi}{\begin{minipage}}
\newcommand{\emi}{\end{minipage}}

\newtheorem{stel}{Theorem}[section]

\newtheorem{lemm}[stel]{Lemma}
\newtheorem{prop}[stel]{Proposition}

\newtheorem{exam}[stel]{Example}
\newtheorem{properties}[stel]{Properties}

\newtheorem{theo}[stel]{Theorem}

\theoremstyle{definition}

\newtheorem{defin}[stel]{Definition}
\newtheorem{rem}[stel]{Remark}
\newtheorem{definlem}[stel]{Definition-Lemma}
\newtheorem{assumption}{Assumptions}

\makeatletter
\newcommand{\myitem}[1]{%
\item[#1]\protected@edef\@currentlabel{#1}%
}
\makeatother
\usepackage{mathtools}
\usepackage{multicol}

\usepackage{MnSymbol,bbding,pifont}

\newcommand{\finsub}[0]{{\subseteq}_{\text{\sf{fin}}}}

\makeatletter
\newcommand{\Matrix}[1]
    {\begin{pmatrix}
      \Matrix@r #1;\@bye;\Matrix@r
     \end{pmatrix}}

\def\Matrix@r #1;{\@bye #1\Matrix@z\@bye\Matrix@s #1,\@bye, }%
\def\Matrix@s #1,{#1\Matrix@t }%
\def\Matrix@t #1,{\@bye #1\Matrix@y\@bye\@firstofone {&#1}\Matrix@t}%
\def\Matrix@y #1\Matrix@t{\\ \Matrix@r }%
\def\Matrix@z #1\Matrix@r {}
\def\@bye  #1\@bye   {}%

\makeatother

\makeatletter
\newsavebox\myboxA
\newsavebox\myboxB
\newlength\mylenA

\newcommand*\xoverline[2][0.75]{%
    \sbox{\myboxA}{$\m@th#2$}%
    \setbox\myboxB\null
    \ht\myboxB=\ht\myboxA%
    \dp\myboxB=\dp\myboxA%
    \wd\myboxB=#1\wd\myboxA
    \sbox\myboxB{$\m@th\myov{\copy\myboxB}$}
    \setlength\mylenA{\the\wd\myboxA}
    \addtolength\mylenA{-\the\wd\myboxB}%
    \ifdim\wd\myboxB<\wd\myboxA%
       \rlap{\hskip 0.5\mylenA\usebox\myboxB}{\usebox\myboxA}%
    \else
        \hskip -0.5\mylenA\rlap{\usebox\myboxA}{\hskip 0.5\mylenA\usebox\myboxB}%
    \fi}
\makeatother
\usepackage{bm}
\usepackage{dsfont}
\usepackage{graphicx}

\newdimen\BBthick    
\newdimen\BBgap      
\newdimen\BBsep      
\newdimen\BBinset    

\makeatletter
\newcommand{\bbbarsop}[1]{\mathop{\mathpalette\bbbarsop@{#1}}}
\newcommand{\bbbarsop@}[2]{%
  \setbox0=\hbox{$\m@th#1 #2$}%
  \dimen0=\wd0   
  \dimen2=\ht0   
  \dimen4=\dp0   
  \dimen6=\dimexpr\dimen0-2\BBinset\relax
  \ooalign{%
    \hfil\box0\hfil\cr
    \hfil\raise\dimexpr\dimen2-\BBsep\relax
      \hbox{\vrule width \dimen6 height \BBthick depth 0pt}\hfil\cr
    \hfil\raise\dimexpr\dimen2-\BBsep-(\BBthick+\BBgap)\relax
      \hbox{\vrule width \dimen6 height \BBthick depth 0pt}\hfil\cr
    \hfil\raise\dimexpr-\dimen4+\BBsep\relax
      \hbox{\vrule width \dimen6 height \BBthick depth 0pt}\hfil\cr
    \hfil\raise\dimexpr-\dimen4+\BBsep+(\BBthick+\BBgap)\relax
      \hbox{\vrule width \dimen6 height \BBthick depth 0pt}\hfil\cr
  }%
}
\makeatother

\newcommand{\norm}[1]{\bigthickbar \hspace{0.1ex} #1 \hspace{0.1ex} \bigthickbar}

\usepackage{amsmath,amssymb,tikz} 

\newcommand{\pps}{\mathord{\begin{tikzpicture}[baseline=0ex, line width=1.5, scale=0.11]
\draw (1,0) -- (1,2);
\draw (0,1) -- (2,1);
\end{tikzpicture}}}

\newcommand{\mms}{\mathord{\begin{tikzpicture}[baseline=0ex, line width=1.5, scale=0.11]
\draw (0,0.6) -- (1.5,0.6);
\end{tikzpicture}}}

\newcommand*{\myov}[1]{\overbracket[1.40pt][-2.0pt]{\hspace{0.15em}#1\hspace{0.15em}}}

\newcommand{\bigthickbar}{\rule[-0.43ex]{1.5pt}{2.3ex}} 

\usepackage{soul}

\setul{0.20ex}{0.10ex}

\newcommand\restr[2]{{
  \left.\kern-\nulldelimiterspace 
  #1 
  \littletaller 
  \right|_{#2} 
  }}

\newcommand{\littletaller}{\mathchoice{\vphantom{\big|}}{}{}{}}
\usepackage{ragged2e}
\title{Inner products for strongly regular near-vector spaces and duality for finite dimensional near-vector spaces}
\author{L Boonzaaier, S. Marques and D. Moore} 
\address[Leandro Boonzaaier]{rain (Pty) Ltd, Cape Quarter, Somerset Road, Green Point, 8005,South Africa}
\email{leandro.boonzaaier@rain.co.za}
\address[Sophie Marques]{Department of Mathematical Sciences, Stellenbosch University, Stellenbosch 7600, South Africa, and National Institute for Theoretical and Computational Sciences, South Africa}
\email{smarques@sun.ac.za}
\address[Daniella Moore]{Department of Mathematical Sciences, Stellenbosch University, Stellenbosch 7600, South Africa}
\email{dmoore@sun.ac.za}
\begin{document}
\maketitle

\begin{abstract}
In this paper we develop a duality theory for all finite–dimensional near–vector spaces and introduce a notion of inner product tailored to the broad and natural class of \emph{strongly regular near–vector spaces}.  
This generalized construction extends the classical inner product beyond the classical framework, yielding rich families of examples on \emph{multiplicative near–vector spaces}.  
Within this setting, several familiar norms—such as those that fail to produce Hilbert spaces in the classical sense—emerge naturally as genuine inner–product–type norms.  
A further contribution is the extension of the theory of \emph{generalized (weighted) means} to arbitrary complex datasets.  
This extension unifies and generalizes the classical power and geometric means, carrying them beyond the domain of positive reals.  
\end{abstract}

\tableofcontents


\section*{Acknowledgement}
We are deeply grateful to Karin-Therese Howell for her insightful discussions, without which this paper would not have seen the light of day.

\section{Introduction}
\justifying
\noindent André introduced the notion of near-vector spaces (see \cite{Andre}). Over the years, researchers have developed the theory in multiple directions (see, for instance, \cite{LSD2024,LSD2025,DeBruyn,HCW,How,HowMar,HowMey,HowRab,Howellspanning,DMZ, Dani,  Rabie, Wessels}). By analogy with classical linear algebra—where duality and inner product spaces knit together algebra, geometry, and analysis with powerful applications in functional analysis—a natural question is whether one can extend the notions of duality and inner products to the setting of near-vector spaces. The present paper undertakes this task, at least in part.

Our main observation is that a suitably more flexible definition of inner product, adapted to near-vector spaces, allows certain classical norms that are \emph{not} induced by any Hilbert structure in the usual linear setting to become Hilbertian when viewed on appropriately chosen near-vector space structures. In particular, by changing the underlying near-vector space structure while preserving enough algebraic control, we exhibit norms that admit inner products in this generalized sense. As a further application, we extend the theory of generalized means to families of complex numbers, obtaining a coherent framework that includes the classical power means as a special case and continues to make sense for complex parameters.

Inner products play a central role in the structure of many real world applications, particularly in the field of data science \cite{Schenck, Murphy}. In many applications, data points are represented as points in a high-dimensional (dependent on the number of features characterizing the data) vector space \cite{Brunton}. The inner product is then used as a means to compare the similarity of data points. For example, in Natural Language Processing (NLP), texts are converted to vectors and the semantic similarity of texts are measured by the cosine of the angle between two vectors. If the angle is small, the cosine is close to 1 and the texts are considered semantically similar. This is known as cosine-similarity in the literature \cite{Sohangir}. Inner products, and related ideas such as orthogonality, play a central role in techniques such as Principle Component Analysis (PCA), where one tries to construct a new basis in which the data becomes uncorrelated \cite{Brunton, James}. Inner products also play an important role in more advanced data science techniques such as neural networks \cite{Nielsen}. For each neuron, one computes $ \langle w_{i}, x \rangle + b$ which involves computing the inner product between the neuron's weight vector and the input vector (data point) \cite{Nielsen}. These are but two examples of the applications of inner products in data science. In essence, inner products provide geometric notions that underpin modern data science methods. However, data is usually interpreted as vectors in spaces with linear structures, which is often an idealization rather than a reflection of the underlying structure of real world data. Any non-linearities inherent in the data structure generally introduce computational complexity 
\cite{Brunton}. For this reason, near-vector spaces may offer a more effective framework for handling data with non-linear structures. Near-vector spaces retain enough algebraic structure to define operations such as addition and scalar multiplication, whereas being more flexible than vector spaces. This flexibility makes them well-suited for modelling spaces in which the data does not exhibit strictly linear behaviour. The introduction of inner products on near-vector spaces thus provides a natural extension of classical geometric tools into this more general setting. By defining such inner products, we obtain generalized notions of norm, angle, and orthogonality that can operate within structures that are not governed by strict linearity.

After a brief preliminaries section, Section~3 introduces a notion of inner product and the associated norm on strongly regular near-vector spaces. Section~4 recalls the classical correspondences: every inner product induces a norm, and conversely a norm arises from an inner product if and only if it satisfies the parallelogram law. Section~5 develops orthogonality in this setting. Section~6 shows that the dual of any finite-dimensional near-vector space is again a near-vector space; in particular, we study the inner-product dual, annihilators, and orthogonal complements. Section~7 defines an angle determined by inner product on strongly regular near-vector spaces. Section~8 constructs a broad family of inner products on multiplicative near-vector spaces. Section~9 uses these constructions to recover the $\ell^p$, $\ell^{p,q}$, and $\mathcal{L}^p$ norms—including the cases $p=\infty$—via suitable inner products on specific multiplicative near-vector spaces. Finally, Section~10 leverages these limiting constructions to define generalized means for families of complex numbers. Generalized means play a foundational role across data science, physics, and engineering—appearing in regularization schemes, loss functions, distance metrics, and sparse-signal processing; see, e.g., \cite{Brunton,James,Murphy}.

\section{Preliminary material}
We begin this section by recalling the definition of a scalar group. This definition is needed to define the most general notion of a near-vector space later on in this section.

\begin{defin}\cite[Definition 1.1]{MarquesMoore} \label{scalargroup}
A {\sf scalar group} \( F \) is a tuple \( F = (F, \cdot, 1, 0, -1) \) where \((F, \cdot, 1)\) is a monoid, \(0, -1 \in F\), satisfying the following conditions: 
\begin{itemize}
 \vspace{0cm}    \item \(0 \cdot \alpha = 0 = \alpha \cdot 0\) for all \(\alpha \in F\);
    \item \(\{ \pm 1 \}\) is the solution set of the equation \(x^2 = 1\);
    \item \((F \setminus \{ 0 \}, \cdot, 1)\) is a group.
\end{itemize}
\vspace{0cm} 
For all \(\alpha \in F\), we denote \(-\alpha\) as the element \((-1) \cdot \alpha\).
\end{defin}
In this paper, it happens that we consider $F$ with several binary operations $\cdot_i$ defined on it, where $i\in I$ for some index set $I$. Given an element $\alpha\in F$, we denote the composition of $\alpha$ with itself $n$ times with respect to the operation $\cdot_i$, $\alpha^{n_i}$ and when $\alpha\in F^*$, we write $(\alpha^{-1})^{n_i} = \alpha^{-n_i}$. 

Near-fields play a central role in the theory of near-vector spaces; for completeness, we briefly recall the definition.

\begin{defin}\cite[Definition 3.2]{LSD2024}
    Let $(F, +, \cdot)$ be a near-field. An element \(\gamma \in F\) is {\sf right distributive} if, for every \(\alpha, \beta \in F\), \((\alpha + \beta)\gamma = \alpha \gamma + \beta \gamma\). We denote by \(F_{\fr{d}}\) the set of all right distributive elements of \(F\).
\end{defin}

Our main examples and applications arise from multiplicative automorphisms. Following \cite{BM2024}, we briefly recall the relevant definitions and auxiliary notions.

\begin{defin} \label{multauto}  
A {\sf multiplicative automorphism $\sigma$ of the near-field $(F, +, \cdot)$} is defined as a monoid automorphism of $(F, \cdot)$. In other words, $\sigma$ satisfies the following properties:  
\begin{itemize} 
\vspace{0cm}  \item $\sigma(1) = 1$, and  
\item $\sigma(\alpha \beta) = \sigma(\alpha) \sigma(\beta)$ for all $\alpha, \beta \in F$.
\end{itemize} 
\end{defin}

\begin{definlem}\cite[Definition-Lemma 3.4]{LSD2025}  
\label{induced}  
Let $\sigma$ and $\rho$ be multiplicative automorphisms of $F$.  
\begin{enumerate}  
    \vspace{0cm}  \item We define the \textsf{addition associated with $\sigma$} as a binary operation $+_\sigma$ on $F$, for all $\alpha, \beta \in F$, as  
    \[
    \alpha +_\sigma \beta = \sigma^{-1}(\sigma(\alpha) + \sigma(\beta)).
    \]  
    The triple $(F, +_\sigma, 0)$ forms an abelian group.  
    If the near-field is denoted by $F_u$ for some $u \in Q(V)^*$, this addition is written as $+_{u, \sigma}$.     Additionally, for $n \in \mathbb{N}$, we define  
    \[
    {}^{\sigma \! \!}\sum_{k=1}^{n} \alpha_k = \alpha_1 +_\sigma \alpha_2 +_\sigma \cdots +_\sigma \alpha_n.
    \]
  
 \item We define the \textsf{near-field induced by $\sigma$}, denoted by $F_{\sigma}$, as the near-field $(F, +_{\sigma}, \cdot)$. 
 
  \item We define the \textsf{scalar multiplication induced by $\rho$} as the scalar multiplication $\cdot_\rho$ on $F$, for all $\alpha, \beta \in F$, as 
    \[
    \alpha \cdot_\rho \beta = \rho(\alpha) \cdot \beta.
    \]  
    This operation $\cdot_\rho$ defines a left monoid action of $F$ on itself.  
\end{enumerate}    
\end{definlem}
In this paper, we develop an inner product theory over scalar groups that we call \(F\)\nobreakdash-like lines. We begin with the definition of an \(F\)\nobreakdash-like line.

\begin{defin}
\label{line}
    Let $F\in \{\mathbb{R},\mathbb{C}\}$, $+$ and $\cdot$ be the usual binary operations on $F$ and $\pps $ and $\bullet$ be binary operations on $F$. We say that $((F, \pps,\bullet),\psi)$ is an {\sf $F$-like line} if $(F,\bullet,\pps)$ is a field and the map 
    \[\psi: (F,+,\cdot) \rightarrow (F,\pps,\bullet)\]
    is a field isomorphism.
    
 For $m \in \mathbb{Z}$, $\alpha\in F$ we set:
\begin{multicols}{3}
    \item[$1.$] $\mathbf{m}:=\psi(m)$;
    \item[$2.$]\label{abs} $\myov{\alpha}:= \psi(\overline{\alpha'}) $; 
    \item[$3.$]\label{plus} $\mathbb{R}^{\pps}:=\psi(\mathbb{R}^{+})$;
    \item[$4.$]\label{leq} $\gamma \pmb{\leq} \delta$ if $\delta \mms \gamma \in \mathbb{R}^{\pps}$;
    \item[$5.$]\label{norm} $\norm{ \alpha}:= \psi(|\alpha'|)$;
    \item[$6.$]\label{i} $\mathbf{i}:= \psi(i)$, 
\end{multicols}

where $\gamma,\delta,  \in F$ and $\alpha' \in F$ such that $\alpha = \psi(\alpha')$. 
 
 We denote 
\vspace{-0cm}\begin{itemize}
        \item $x^{\bm{m}} = x \bullet \ldots \bullet x$, $m$ times, when $m \in \mathbb{N}$, and 
        \item $x^{\bm{m}} = x^{\mms(\mms \bm{m})}$ when $m \in \mathbb{Z} \setminus \mathbb{N}$;
        \item given $a \in \mathbb{R}^{\pps}$, $x^{\bm{2}}= a$ has a unique solution in $\mathbb{R}^{\pps}$. We denote this unique solution by $x = a^{\bm{2^{-1}}}$.
    \end{itemize}
    \end{defin}
    \begin{exam}\label{exam1}
Let \(F\in\{\mathbb{R},\mathbb{C}\}\) and let \(\sigma:F\to F\) be a
multiplicative automorphism of \(F\).
\[
\sigma^{-1}:\ (F,+,\cdot)\longrightarrow(F,+_{\sigma},\cdot),\qquad
\alpha\longmapsto\sigma^{-1}(\alpha)
\]
is a field isomorphism (multiplication is preserved by multiplicativity of \(\sigma\),
and the transported addition is preserved by construction). Hence
\(((F,+_{\sigma},\cdot), \sigma^{-1})\) is an \(F\)\nobreakdash-like line.

\smallskip
 On \(F=\mathbb{R}\), \(\sigma(x)=x^{5}\) is a continuous
automorphism of \(\mathbb{R}^\times\), yielding a nontrivial \(+_{\sigma}\).
On \(F=\mathbb{C}\), complex conjugation \(\sigma(z)=\overline{z}\) is a (field, hence
multiplicative) automorphism and gives another instance. A classification of continuous multiplicative automorphisms of \(\mathbb{R}\) and \(\mathbb{C}\) can be found in~\cite{BM2024}.
    \end{exam}

\noindent We collect a few properties that are immediate from \(\psi\) being a field isomorphism; they will be used in the proofs that follow. For a proof in the classical setting, see \cite[Thm.~1.37]{Rudin2}.

\begin{properties}
\label{properties}
Let $((F, \pps,\bullet),\psi)$ be an $F$-like line. Then for all $\alpha, \beta, \gamma \in F$, we have:
\vspace{0cm}\begin{enumerate}
        \item\label{m} 
    $\mathbf{m} \bullet \alpha = \underbrace{\alpha \pps \ldots \pps \alpha}_{\text{$m$ times}}$ and $\mathbf{m^{-1}} \bullet (\underbrace{\alpha \pps \ldots \pps \alpha}_{\text{$m$ times}}) = \alpha$.
    \item\label{-1} for all $\alpha, \beta \in F$, we have $-1\bullet \psi(\alpha) = \mms \psi(\alpha)$ and $(-\alpha)\bullet \beta = \mms (\alpha \bullet \beta)$;
    \item\label{double} $\myov{\myov{\alpha}} = \alpha$;
       \item\label{preserve} $\myov{\alpha} \pps \myov{\beta}  = \myov{\alpha \pps \beta}$ and $\myov{\alpha \bullet \beta} = \myov{\alpha} \bullet \myov{\beta}$;
    
    \item\label{minus} $\mms\myov{\alpha} = \myov{\mms\alpha}$;
    \item\label{psileq} if $\alpha \pmb{\leq} \beta$, then $\psi(\alpha) \pmb{\leq} \psi(\beta)$;
    \item\label{squareroot}
    if $\alpha, \beta \in \mathbb{R}^{\pps}$ and $\alpha^{\bm{2}} \pmb{\leq} \beta^{\bm{2}}$, then $\alpha \pmb{\leq} \beta$; 
    
    \item\label{leq1} if $\alpha \pmb{\leq} \beta$, we have $\gamma \bullet \alpha \pmb{\leq} \gamma \bullet \beta$; 
    \item\label{triangleineq} $\norm{\alpha \bullet \beta} = \norm{\alpha} \bullet \norm{\beta}$ and $\norm{\alpha \pps \beta} \pmb{\leq} \norm{\alpha}\pps \norm{\beta}$;
    \item\label{square} $\alpha \bullet \myov{\alpha} = \norm{\alpha}^{\bm{2}}$.
    \end{enumerate}
\end{properties}

\begin{properties}
\label{invertible}
Let $((\mathbb{C}, \pps, \bullet),\psi)$ be a $\mathbb{C}$-like line. Then, as in the usual case, we have 
\begin{multicols}{3}
\begin{description} 
    \item[$1.$] $\mathbf{i}$ is invertible and $\mathbf{i^{-1}} = \mms\mathbf{i}$;
    \item[$2.$] $\mathbf{i}^{\bm{2}} = -1$ and $\mathbf{i}^{\bm{-2}} = -1$;
    \item[$3.$] $\mathbf{i}^{\bm{3}}= \mms\mathbf{i}$ and $\mathbf{i}^{\bm{-3}}= \mathbf{i}$. 
    \end{description}
\end{multicols}
\end{properties}

To speak of “real” and “imaginary” parts inside an \(F\)-like line, we simply
transport the classical notions along the fixed field isomorphism \(\psi\).

\begin{defin}
Let \(((\mathbb{C},\pps,\bullet),\psi)\) be a \(\mathbb{C}\)-like line. For \(\alpha\in\mathbb{C}\), write \(\alpha=\psi(\alpha')\) with the unique \(\alpha'\in F\).
We define the real and imaginary parts in \((\mathbb{C},\pps,\bullet)\) by
\[
\mathbf{Re}(\alpha):=\psi\big(\operatorname{Re}(\alpha')\big),
\qquad
\mathbf{Im}(\alpha):=\psi\big(\operatorname{Im}(\alpha')\big).
\]
\end{defin}


We record several identities for \(\mathbf{Re}\) and \(\mathbf{Im}\) that remain valid with the
transported operations \(\pps\) and \(\bullet\); these will be used in subsequent proofs.

\begin{properties}\label{properties2}
Let \(((\mathbb{C},\pps,\bullet),\psi)\) be an \(F\)-like line and \(\alpha,\beta\in\mathbb{C}\).
\begin{enumerate}
    \item\label{1} \(\mathbf{Re}(\mms \alpha)=\mms \mathbf{Re}(\alpha)\).
    \item\label{2} \(\displaystyle
      \mathbf{Re}(\alpha)=\bm{2}^{-1}\!\bullet(\alpha \pps \myov{\alpha}),\qquad
      \mathbf{Im}(\alpha)=(\bm{2}\bullet \mathbf{i})^{-1}\!\bullet(\alpha \mms \myov{\alpha}).\)
    \item\label{3} \(\displaystyle
      \alpha=\mathbf{Re}(\alpha)\pps \mathbf{i}\bullet \mathbf{Im}(\alpha).\)
    \item\label{4} \(\displaystyle
      \|\mathbf{Re}(\alpha)\|\;\le\;\|\alpha\|.\) 
    \item\label{5} If \(\alpha\in\mathbb{R}\) and \(\beta\in\mathbb{C}\), then
      \(\mathbf{Re}(\alpha\bullet \beta)=\alpha\bullet \mathbf{Re}(\beta)\).
\end{enumerate}
\end{properties}

We now recall the notion of a near-vector space, together with the auxiliary concepts required for its definition.

\begin{defin}
    \cite[Definition 1.4]{MarquesMoore}
\begin{enumerate}    
\vspace{0cm}  \item Given a scalar group $F$, we say that a {\sf (left) $F$-space} is a pair $(V, \mu)$ where $V$ is an abelian additive group and $\mu: F \times V \rightarrow V$ is a (left) scalar group action, that is \(\mu\) is a (left) monoid action of $(F,\cdot)$ on $(V,+)$ on the (left) by endomorphisms, such that $0$ acts trivially, and $\pm 1$ act as $\pm {\sf id}$. 
\item We say that $W$ is an {\sf $F$-subspace} of $V$ if $W$ is a nonempty subset of $V$ that is closed under addition and scalar multiplication.
\end{enumerate}
\end{defin}

\begin{defin}\cite[Definition 1.5]{MarquesMoore}
\label{quasikernel}
Let $V$ be an $F$-space. We define the {\sf quasi-kernel} of $V$ to be
$$Q(V) = \{u \in V \mid \forall {\alpha, \beta \in F}, \exists {\gamma \in F} \,[\alpha \cdot u + \beta \cdot u = \gamma \cdot u]\}.$$
\end{defin}

\begin{defin}\cite[Definition 1.6]{MarquesMoore}
\label{NVS2}
A {\sf (left) near-vector space over a scalar group $F$} is a triple \((V, F, \mu)\) that satisfies the following properties:
\begin{enumerate}
\vspace{0cm} \item \((V, \mu)\) is an $F$-space; 
    \item the (left) scalar group action \(\mu\) is free;
    \item the quasi-kernel \(Q(V)\) generates \(V\) as an additive group.
\end{enumerate}
\vspace{0cm} 
Any trivial abelian group can be endowed with a near-vector space structure through the trivial action. Such a space is referred to as a {\sf (left) trivial near-vector space over $F$}, denoted as \(\{0\}\). We call the elements of $V$ {\sf vectors}.
\end{defin}
\begin{defin}
    Let   $((V,+_{V}),\cdot_{V})$ be a near-vector space over $(F, \cdot)$.
     Then, for all $\alpha, \beta \in F$, we define $\alpha -_{V} \beta$ to be $\alpha +_{V} (-\beta)$.
\end{defin}

\begin{defin}\cite[Definition 4.7]{Andre}
\label{compatible}
Let $V$ be an $F$-space.
Elements $u,v \in Q(V)^{*}$ are called {\sf compatible} if there is $\lambda \in F^{*}$ such that $u + \lambda v \in Q(V)$.
\end{defin}

\begin{defin}\cite[Definition 4.11]{Andre}
\label{regulardefinition}
    A near-vector space $(V,F)$ is called {\sf regular} if any two elements $u,v \in Q(V)^{*}$ are compatible (in the sense of Definition \ref{compatible}). 
\end{defin}

We now introduce the notion of a \emph{strongly regular} near-vector space. 
Strong regularity will serve as a standing hypothesis for the inner product theory developed in the next section: in particular, it is a necessary condition for the existence of inner products in the sense defined there. Further details will be discussed below.

\begin{defin}
 We say that a near-vector space $V$ over $F$ is {\sf strongly regular} if $Q(V)=V$. 
\end{defin}

\begin{rem}
\label{strongly}
Suppose $V$ is a strongly regular near-vector space over $F$. Then:
\vspace{0cm}\begin{enumerate}
\item $+_u = +_v$, for all $u,v\in Q(V)$. This is by \cite[Theorem 4.2]{HowMar}.
    \item 
 $V$ is regular. Indeed, by 1. and \cite[Hilfssatz 4.8]{Andre}, taking $\lambda = 1$, we get our desired result;
\end{enumerate}
\end{rem}
Several types of bases can be defined for a near-vector space; we briefly recall their definitions.
\begin{defin}\cite[Definition 3.2]{Howellspanning}
Consider an \(F\)-space \(V\). We define the {\sf span of a set \(S\)} as the intersection \(W\) of all non-empty subsets of \(V\) that are closed under addition and scalar multiplication and contain the set \(S\). This span is denoted as \(\operatorname{Span}(S)\).
\end{defin}
\vspace{0cm}We write \(A \finsub B\) to indicate that \(A\) is a finite subset of \(B\). (see \cite[Definition 1.13]{MarquesMoore}).

\begin{lemm}\cite[Lemma 2.3]{MarquesMoore}
\label{spanlin}
Let \(V\) be an \(F\)-space, and \(S\subseteq V\). Then, the set \(\operatorname{Span} (S)\) can be expressed as follows:
\begin{align*}
     \operatorname{Span} (S)= & \left\{\sum_{a \in A} \left(\sum_{i=1}^{n_a} \alpha_{a,i} \cdot a \right) \middle| A \finsub S, n_a \in \mathbb{N}, \alpha_{a,i} \in F, \forall a\in A, \forall i \in \{ 1, \cdots ,n_a\}  \right\}\\
= & \sum_{s\in S}   \operatorname{Span} (s).
\end{align*}
\end{lemm}

\begin{defin}\cite[Definition 2.6]{MarquesMoore}
\label{linindep} 
In an \(F\)-space \(V\), let \(S \subseteq V\). \(S\) is considered {\sf linearly independent} if \(0 \notin S\) and for any non-empty subset \(A \finsub S\), along with \(n_a\in \mathbb{N}\) and \(\alpha_{a,i} \in F\) where \(a \in A\) and \(i \in \{ 1, \cdots , n_a\}\), if \(\sum_{a \in A} \left( \sum_{i=1}^{n_a} \alpha_{a,i} \cdot a \right) = 0\), then \(\sum_{i=1}^{n_a} \alpha_{a,i}\cdot a = 0\) for all \(a \in A\).   
\end{defin}
\begin{defin}\cite[Definition 2.2]{MarquesMoore}
Let \(V\) be an \(F\)-space and \(W\) be an \(F\)-subspace of \(V\). A subset \(S\) of \(V\) is called a {\sf generating set for \(W\)} if \(W \subseteq  \operatorname{Span}(S)\). 
\end{defin}
\begin{defin}\cite[Definition 2.14]{MarquesMoore}
\label{basis}
Let \(V\) be an \(F\)-space. 
\begin{enumerate} 
\item The set \(S\) is called an {\sf $F$-basis for \(V\)} (or simply {\sf basis} when there is no confusion) if \(S \subseteq V\), \(S\) is a generating set for \(V\), and \(S\) is  linearly independent. 
\item A {\sf scalar \(F\)-basis} (or simply {\sf scalar basis} when there is no confusion) is a set that is a basis whose elements are scalar. 
\item $V$ is said to be {\sf finite-dimensional} if $V$ admits a finite scalar $F$-basis.
\item When \(V\) is a finite-dimensional near-vector space, we say that a basis \(B\) is a {\sf minimal (resp. maximal) basis for \(V\)}, if \(|B|= \operatorname{min} \{ |C| \mid C \text{ basis of } V\}\) (resp. \(|B|= \operatorname{max} \{ |C| \mid C \text{ basis of } V\}\)). We define the minimal (resp.\ maximal) dimension of a near-vector space to be the cardinality of a minimal (resp.\ maximal) \(F\)-basis (when these extrema exist).
\end{enumerate}
\end{defin}
For completeness, we recall the general notion of isomorphism in the setting of near-vector spaces.
\begin{defin}{\cite[Definition 3.2]{HowMey}}\label{homo}
Let $((V_{1},+_{1}),\cdot_1)$ and $((V_{2},+_{2}),\cdot_2)$ be near-vector spaces over $(F_{1},*_{1})$ and over $(F_{2}, *_{2})$ respectively.
\begin{enumerate}
   \vspace{0cm} 
    \item A pair \((\theta, \eta)\) is called a {\sf homomorphism} of near-vector spaces if  \(\theta : (V_1, +_{1}) \rightarrow (V_2, +_{2})\) is an additive homomorphism and  \(\eta : (F_1^{*}, *_{1}) \rightarrow (F_2^{*}, *_{2})\) is a monoid isomorphism such that \(\theta(\alpha \cdot_{1} u) = \eta(\alpha) \cdot_{2} \theta(u)\) for all \( u \in V_1 \) and \(\alpha \in F_1^{*}\). A pair \((\theta, \eta)\) is called a {\sf isomomorphism}  of near-vector spaces if \((\theta, \eta)\) is a homomorphism of near-vector spaces and \( \theta \) is an isomorphism. 
    \item If \(\eta = \operatorname{Id}\), we say that \(\theta\) is an \( F \)-{\sf linear map} where $F = F_{1} = F_{2}$. 
    \item If \(\eta = \operatorname{Id}\) and $\theta$ is a bijection, we say that \(\theta\) is an {\sf \( F \)-isomorphism} of near-vector spaces. 
\end{enumerate}
\end{defin}
For completeness, we review the notion of multiplicative near-vector spaces; this will be the primary setting in which our applications are formulated.
\begin{definlem}\cite[Definition-Lemma 3.11]{LSD2025}
\label{multautnvs}
Let \((F, +, \cdot)\) be a near-field, and let \(I\) be an index set. Consider two tuples of multiplicative automorphisms, \(\boldsymbol{\sigma} = (\sigma_i)_{i \in I}\) and \(\boldsymbol{\rho} = (\rho_i)_{i \in I}\).  
\begin{enumerate} 
    \vspace{0cm} \item We define the {\sf multiplicative near-vector space associated with \((\boldsymbol{\sigma}, \boldsymbol{\rho})\)}, denoted by \(F^{\boldsymbol{\sigma}, \boldsymbol{\rho}}\), as the near-vector space  
    \[
    F^{\boldsymbol{\sigma}, \boldsymbol{\rho}} = \left\{ (\alpha_i)_{i\in I} \in \prod_{i\in I} F^{\sigma_i, \rho_i} \mid (\alpha_i)_{i\in I} \asymp \mathbf{0} \right\}.
    \]  
    where the near-vector space structure is given componentwise by $F^{\sigma_i, \rho_i}$, for all $i \in I$. Additionally, for \( k \in \{1, \dots, n\} \), we define  
\[
 {}^{\boldsymbol{\sigma} \! \! }\sum_{k=1}^{n}(\alpha_{k_i})_{i \in I} = \left( {}^{\sigma_{i}\! \! }\sum_{k=1}^{n} \alpha_{k_i}\right)_{i \in I},
\]  
and we adopt the convention \( F^{0,0} = \{0\} \).
    \item The {\sf canonical basis of \(F^{\boldsymbol{\sigma}, \boldsymbol{\rho}}\)} is given by the set \(\mathcal{E} = \{ {\bf e}_j \mid j \in I\}\), where \({\bf e}_j = (\delta_{j,i})_{i\in I}\). For all $j \in I$ and $\gamma \in F^{*}$, we have
    $$+_{\gamma \cdot_{\boldsymbol{\rho}} {\bf e}_j}=+_{\sigma_{j} \circ \rho_{j}\circ \varphi_{\gamma}  }.$$
\end{enumerate}
\end{definlem}

\begin{defin}\cite[Definition 3.15]{LSD2025}
\label{multNVS}
     Let $(F,+,\cdot)$ be a near-field. We say that a near-vector space $V$ over $F$ is {\sf multiplicative } if it is $F$-isomorphic as near-vector spaces to $(F^{\boldsymbol{\sigma},\boldsymbol{\rho}},F)$ for some tuples of multiplicative automorphisms $\boldsymbol{\sigma} = (\sigma_i )_{i\in I}$ and $\boldsymbol{\rho}=( \rho_i )_{i\in I}$. 
\end{defin}

\section{Defining inner products and norms on strongly regular near-vector spaces}

We begin this section by introducing an inner product and a norm in the setting of strongly regular near-vector spaces.


\begin{defin}
\label{innerproductnvs}
    Let $((F,\pps,\bullet), \psi)$ be an $F$-like line and $((V,+_{V}),\cdot_{V})$ be a strongly regular near-vector space over $(F,\cdot)$. We say that the map 
    \begin{equation*}
        \langle - , - \rangle: V \times V \rightarrow F
    \end{equation*}
    is an {\sf inner product on $(((V,+_{V}),\cdot_{V}), (F, \cdot))$ over $(F,\pps,\bullet)$} if the following assertions hold for all $\alpha \in F$ and $u, v,w \in V$: \vspace{0cm}
    \begin{enumerate}
        \item[$I1.$] $\langle u,v \rangle = \myov{\langle v,u\rangle}$;
        \item[$I2.$] $\langle \alpha \cdot_{V}u, v\rangle = \alpha \bullet\langle u,v\rangle$;
        \item[$I3.$] $\langle u+_{V}v,w\rangle = \langle u,w\rangle\pps \langle v, w \rangle$;
        \item[$I4.$] $\langle u,u \rangle \in \mathbb{R}^{\pps}$, for all $u\in V^*$;
        \item[$I5.$] $\langle u, u \rangle = 0$ if and only if $u=0$. 
    \end{enumerate}\vspace{0cm} 
     We say that $V$ together with $\langle -,- \rangle$ is an {\sf inner product near-space over $(F,\pps,\bullet).$}
     When only (I1)--(I4) are satisfied, we say that $\langle -,- \rangle$ is an {\sf indefinite inner product on $(((V,+_V),\cdot_V),(F,\cdot))$ over $(F, \pps, \bullet)$}.

\end{defin}

In the following example, instead of writing  $+_{\epsilon_{p}}$, we simply write $+_{p}$.

\begin{exam}
\label{inner_example}
    Let $F \in \{\mathbb{R},\mathbb{C}\}$, $(F,\cdot)$ be a scalar group and $((F,+_{5},\cdot), \epsilon_{5^{-1}})$ be the $F$-like line as defined in Example \ref{exam1}. We define an action of endomorphism of $F$ on $F^{2}$ as follows:
    \[ \alpha \odot (u_1,u_2) = (\alpha \cdot u_1, \alpha^{3}\cdot u_2) \quad (\alpha \in F, (u_1,u_2) \in F^{2}).\]
    We also define a new addition on $F^{2}$ as follows:
    \[ (u_1,u_2) \oplus (v_1,v_2) = (u_1+v_1, (u_2^{1/3}+v_2^{1/3})^{3})\quad ((u_1,u_2),(v_1,v_2) \in F^{2}).\]
    Then, 
    $((F^{2},\oplus),\odot)$ is a strongly regular near-vector space over $(F,\cdot)$. We can define an inner product $\langle -,- \rangle$ on $(((F^{2},\oplus),\odot),(F,\cdot))$ over $(F,+_5,\cdot)$ as follows: 
    \[ \langle (u_1,u_2),(v_1,v_2) \rangle = u_1\cdot \overline{v_1} +_{5} (u_2\cdot \overline{v_2})^{1/3} \quad ((u_1,u_2),(v_1,v_2) \in F^{2}).\]
\end{exam}
\begin{rem}
The strong regularity hypothesis on \(\big((V,+_{V}),\cdot_{V}\big)\) over \((F,\cdot)\) is a
necessary condition for the existence of an inner product (in the sense adopted here).
Indeed, let \(\alpha,\beta\in F\) and \(u\in Q(V)^{\ast}\). By (I2)–(I3) we have
\begin{align*}
\big\langle \alpha\cdot_{V}u +_{V} \beta\cdot_{V}u,\, u \big\rangle
&= \big\langle \alpha\cdot_{V}u,\,u \big\rangle \;\pps\; \big\langle \beta\cdot_{V}u,\,u \big\rangle \\[2pt]
&= \alpha \bullet \langle u,u\rangle \;\pps\; \beta \bullet \langle u,u\rangle \\[2pt]
&= (\alpha \pps \beta)\,\bullet \langle u,u\rangle,
\end{align*}
where the last equality uses the distributivity of \(\bullet\) over \(\pps\).
On the other hand, by the defining property of \(+_{u}\) and again (I2)–(I3),
\begin{align*}
\big\langle \alpha\cdot_{V}u +_{V} \beta\cdot_{V}u,\, u \big\rangle
= \big\langle (\alpha +_{u} \beta)\cdot_{V}u,\, u \big\rangle
= (\alpha +_{u} \beta)\,\bullet \langle u,u\rangle.
\end{align*}
Since \(\langle u,u\rangle\) is invertible by (I5), we can cancel $\langle u,u \rangle$ to obtain
\[
\alpha \pps \beta \;=\; \alpha +_{u} \beta \qquad \text{for all }\alpha,\beta\in F.
\]
Thus, \(+_{u}\) coincides with \(\pps\) for every \(u\in Q(V)^{\ast}\); in particular, the induced
addition is independent of \(u\), i.e.\ \(V\) is strongly regular. This shows that strong
regularity is necessary for the existence of an inner product in our setting.
\end{rem}

We record several immediate consequences of the inner–product axioms. 
In the proofs of \ref{inner1} and \ref{inner5} we use Properties~\ref{properties} ~\ref{preserve}. 
The remaining verifications parallel the classical case; see, e.g., \cite[p.~129]{Kreyszig}.
\begin{properties}\label{innerprodproperties}
Let $((F,\pps,\bullet), \psi)$ be an $F$-like line and $(V, \langle -,- \rangle)$ be an inner product near-space over $(F, \pps,\bullet)$ as defined in Definition \ref{innerproductnvs}. Let $\alpha, \beta \in F$ and $u,v,w \in V$. Then: 
\begin{multicols}{2}
\begin{enumerate}
    \item\label{inner1}  $\langle u, \alpha \cdot_{V} v\rangle = \myov{\alpha} \bullet\langle u,v\rangle$; 
    \item\label{inner2} $\langle u,w+_{V}v\rangle = \langle u,w\rangle\pps \langle u, v\rangle$;
    \item\label{inner3} $\langle \alpha \cdot_V u+_{V} \beta\cdot_V v,w\rangle = \alpha \bullet \langle u,w\rangle\pps \beta \bullet \langle v, w \rangle$;
    \item\label{inner4} $\langle u,\alpha \cdot_V w+_{V} \beta\cdot_V v\rangle = \alpha \bullet \langle u,w\rangle\pps \beta \bullet \langle u, v\rangle$; 
    
    \item\label{inner5} $\langle u,0\rangle = \langle 0,u\rangle = 0$.  
\end{enumerate}
\end{multicols}
\end{properties}

\vspace{-0.5cm}\noindent We now define a norm for strongly regular near-vector spaces.

\begin{defin}
\label{NVSnorm}
    Let $((F,\pps,\bullet), \psi)$ be an $F$-like line. Let $((V,+_{V}),\cdot_{V})$ be a strongly regular near-vector space over $(F,\cdot)$. 
    We say that the map
$\|-\|: V \rightarrow \mathbb{R}$ is a {\sf norm on $((V,+_{V}),\cdot_{V})$ over $(F,\pps,\bullet)$} if the following conditions hold for all $\alpha \in F$ and $u,v \in V$: 
\begin{enumerate}\vspace{0cm}
    \item[$N1.$] $\|\alpha \cdot_{V} u\| = \norm{\alpha}\bullet\|u\|$;
    \item[$N2.$] $\|u+_{V}v\| \pmb{\leq} \|u\| \pps \|v\|$;
    \item[$N3.$] $\|u\| \in \mathbb{R}^{\pps}$;
    \item[$N4.$] $\| u\| = 0$ if and only if $u = 0$. 
   \vspace{0cm}\end{enumerate}
    When only (N1)-(N3) are satisfied, we say that $\|-\|$ is a {\sf semi-norm on $((V,+_{V}),\cdot_{V})$ over $(F,\pps,\bullet)$}.

\end{defin}

The following properties of the norm are proven exactly as in \cite[Theorem 6.4]{LA}, using the field isomorphism to transport the classical arguments.
\begin{properties}\label{normproperties}
Let $((F,\pps,\bullet), \psi)$ be an $F$-like line, $((V,+_{V}),\cdot_{V})$ be a near-vector space over $(F,\cdot)$ and $(V,\|-\|)$ be a normed near-space over $(F,\pps,\bullet)$. Let $\alpha, \beta \in F$ and $u,v \in V$. Then:
\begin{enumerate}
    \item\label{norm1} $\|u\|\mms\|v\|  \pmb{\leq}  \|u-_{V}v\|$. This follows by $N2.$ and Properties \ref{properties} \ref{psileq};
    \item\label{norm2} $ \|\alpha \cdot_{V} u +_{V} \beta \cdot_{V} v \| \pmb{\leq} \norm{\alpha} \bullet \|u\|\pps \norm{\beta}\bullet\|v\|$. This is by $N1.$ and $N2.$
\end{enumerate}
\end{properties}
\section{From inner products to norms (and vice versa) on strongly regular near-vector spaces}

In the classical setting it is well known that every inner product induces a norm and,
conversely, that every norm satisfying the parallelogram identity arises from a unique
inner product (via the polarization identities). The aim of this section is to establish
the corresponding statements in our framework.

We work over an \(F\)\nobreakdash-like line \(((F,\pps,\bullet),\psi)\) and a strongly regular
near-vector space \(((V,+_{V}),\cdot_{V})\) over \((F,\pps,\bullet)\).
We state only the results needed later, omitting proofs, since the classical arguments
carry over verbatim upon transporting the structure along the field isomorphism \(\psi\).

More precisely, given an inner product \(\langle-,-\rangle\) on \(V\), the map
\[
\|u\|\;:=\;\langle u,u\rangle^{\bm{2^{-1}}}
\]
defines a norm on \(V\). 

The key ingredients are the Cauchy–Schwarz inequality and the
triangle inequality in this setting, which we verify below.

 The following properties can be proven in a similar way as in \cite[Lemma 6.5]{LA}, taking into account the fact that $\psi$ is an isomorphism in Definition \ref{line}. 
   Consider $\alpha \in F$ and $u,v \in V$. Then, 
        
        \begin{enumerate}
            \item $\langle u\pm_{V}\alpha\cdot_{V}v,u \pm_{V}\alpha\cdot_{V}v\rangle = \langle u,u \rangle \pmb{\pm} \bm{2}\bullet \mathbf{Re}(\alpha \bullet \myov{\langle u,v \rangle}) \pps   \norm{\alpha}^{\bm{2}} \bullet  \langle v,v \rangle$
            \item
        $\langle u+_{V}v, u+_{V}v\rangle\pps\langle u-_{V}v, u-_{V}v \rangle = \bm{2}\bullet\langle u, u \rangle \pps \bm{2}\bullet\langle v,v\rangle$ (the parallelogram law).
        \end{enumerate}

\noindent The parallelogram law is a classical result in vector spaces (see \cite[Chapter~6, p.~134]{lang}).
It is used to show that an \(F\)-like line–type inequality holds for inner product near-spaces.
The proof can be carried out in the same way as in the classical case; see, for example,
\cite[p.~137]{Kreyszig}.

To prove the Cauchy–Schwarz inequality in our setting, we follow the classical argument (for example, see \cite[Lemma 3.2-1]{Kreyszig}), transporting the field operations along the isomorphism \(\psi\).

\begin{theo}[Cauchy-Schwarz Inequality]
\label{cauchy}
     Let $((F,\pps,\bullet), \psi)$ be an $F$-like line and $((V,+_V),\cdot_V)$ is a strongly regular near-vector space over $(F,\cdot)$ with inner product $\langle -,- \rangle$ over $(F,\pps,\bullet)$.
 Then for every $u,v \in V$, 
     \[
   \norm{\langle u,v \rangle}^{\bm{2}} \pmb{\leq} \langle u,u \rangle \bullet \langle v,v\rangle.
     \]
 \end{theo}
\noindent Using the Cauchy–Schwarz inequality, the triangle inequality follows exactly as in the classical
vector-space case (see \cite[p.~138]{Kreyszig}), upon transporting the field operations along the
isomorphism \(\psi\).

 \begin{theo}[Triangle Inequality]
 \label{triangle}
 Let $((F,\pps,\bullet), \psi)$ be an $F$-like line, $((V,+_{V}),\cdot_{V})$ be a near-vector space over $(F,\cdot)$ and $(V, \langle -,- \rangle$) be an inner product near-space over $(F,\pps,\bullet)$. Then for every $u,v \in V$, 
     \[
     \langle u+_{V} v, u+_{V} v\rangle^{\bm{2^{-1}}} \pmb{\leq} \langle u,u\rangle^{\bm{2^{-1}}} \pps \langle v,v \rangle^{\bm{2^{-1}}}.
     \]
 \end{theo}
As in the classical setting, every inner product induces a norm. Likewise, on a strongly
regular near-vector space, an inner product gives rise to a norm. The proof (see \cite[p. 136--138]{Kreyszig})
proceeds exactly as in the standard case and uses Theorem~\ref{triangle}; we omit the routine
details.

\begin{definlem}
  \label{normassociated}
    Let $((F,\pps,\bullet), \psi)$ be an $F$-like line. Let $((V,+_{V}),\cdot_{V})$ be a strongly regular near-vector space over $(F,\cdot)$ and $(V, \langle -,- \rangle$) be an inner product near-space over $(F,\pps,\bullet)$. Then, $\langle -,-\rangle^{\bm{2^{-1}}}$ is a norm on $((V,+_{V}),\cdot_{V})$ over $(F,\pps, \bullet)$. We denote such a norm by $\|-\|$.
\end{definlem}

The Jordan--von Neumann theorem \cite{JordanNeumann} proves that any normed space defined an inner spaces as long as the norm satisfies the parallelogram law. 
The proof parallels the classical argument, translated via the field isomorphism of the 
$F$-like line.
\begin{theo}[Characterization of inner product norms]
   Let $((F,\pps,\bullet), \psi)$ be an $F$-like line and $((V,+_V),\cdot_V)$ is a strongly regular near-vector space over $(F,\cdot)$ with a norm $\| -\|$ over $(F,\pps,\bullet)$.
There exists an inner product $\langle -,-\rangle$ on $((V,+_V),\cdot_V)$ over $(F, \pps, \bullet)$ such that $\|u\|=\langle u,u \rangle^{\bm{2}^{\bm{-1}}}$
and with respect to which $V$ is an inner product near-space if and only if
the \emph{parallelogram law} holds:
\[
\|u+_Vv\|^{\bm{2}}\pps \|u-_Vv\|^{\bm{2}}=\bm{2}\bullet \big(\|u\|^{\bm{2}}\pps\|v\|^{\bm{2}}\big)
\quad\text{for all } u,v\in V.
\]
In this case, the inner product is uniquely determined by the norm via the
polarization identity: for any $u,v \in V$, we have
    \begin{enumerate}
        \item When $F = \mathbb{R}$, we have \[\langle u,v \rangle = \bm{4^{-1}} \bullet (\|u+_{V} v{\|^{2}}\mms \|u-_{V} v \|^{2});\]
        \item When $F = \mathbb{C}$, we have  \[\langle u,v \rangle = \bm{4^{-1}}\bullet ^{{}\! \!}\sum_{k=0}^{3}\mathbf{i}^{\bm{-k}}\bullet \|u +_{V} \mathbf{i}^{\bm{k}}\cdot_{V} v\|^{\bm{2}}.\]
    \end{enumerate}
\end{theo}

The following figure illustrates the unit circle, centered on different points, under the norm induced by the inner product defined in Example \ref{inner_example}.       
\begin{multicols}{2}
\begin{center}
\includegraphics[scale=0.4]{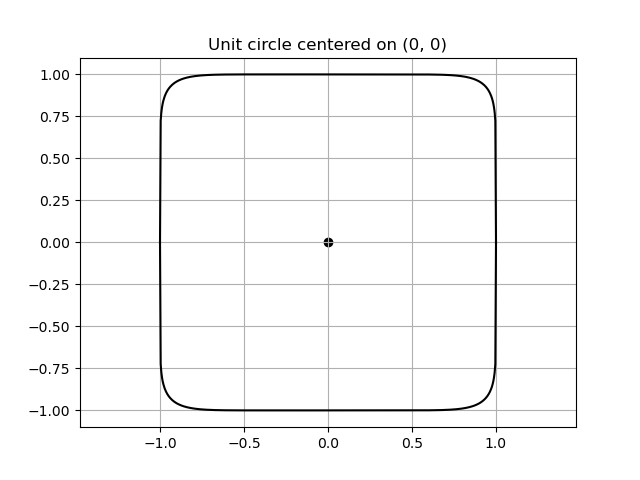}\par
{Unit circle centered on $(0,0)$ }
\end{center}
\columnbreak
\begin{center}
\includegraphics[scale=0.4]{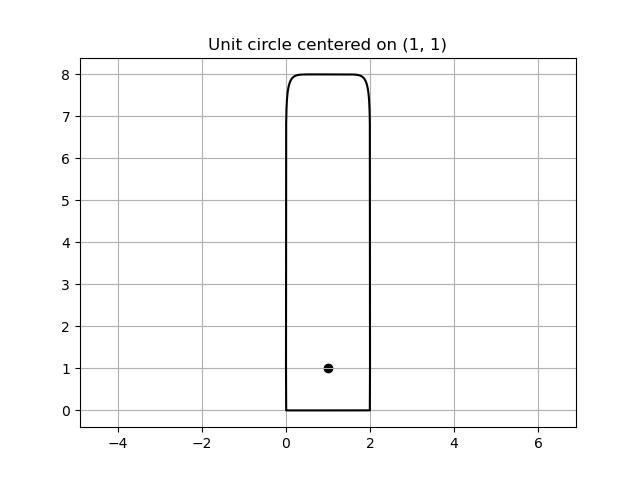}\par
{Unit circle centered on $(1,1)$} 
\end{center}
\end{multicols}

\section{Orthogonality}
The importance of endowing a space with an inner–product structure lies, in particular, in the induced notion of orthogonality. This, in turn, yields orthogonal (and orthonormal) bases and provides a natural generalization of the Pythagorean theorem. This section is devoted to reviewing these foundational results and expressing them in the context of near–vector spaces.
We first formalize the notions of orthogonality and orthonormality within our inner–product near–space setting.
\begin{defin}
Let \(((F,\pps,\bullet),\psi)\) be an \(F\)-like line, \((V,+_{V},\cdot_{V})\) be a strongly regular
near–vector space over \((F,\cdot)\), and \((V,\langle-,-\rangle)\) be an inner–product near–space
over \((F,\pps,\bullet)\).
\begin{itemize}
  \item For \(u,v\in V\), we say that \(u\) and \(v\) are \textsf{orthogonal} if \(\langle u,v\rangle=0\).
        We denote this by \(u\pmb{\perp}v\).
  \item Given \(u_{1},\dots,u_{n}\in V\), the set \(\{u_{1},\dots,u_{n}\}\) is a \textsf{set of orthogonal vectors}
        if \(u_{j}\pmb{\perp} u_{k}\) for all distinct \(j,k\in\{1,\dots,n\}\). It is \textsf{orthonormal} if
        \[
          \langle u_{j},u_{k}\rangle=\delta_{jk}\quad\text{for all }j,k\in\{1,\dots,n\}.
        \]
  \item The set \(\{u_{1},\dots,u_{n}\}\) is an \textsf{orthonormal basis} of \(V\) if it is orthonormal and spans \(V\) as a near-vector space.
\end{itemize}
\end{defin}

The following orthogonality properties remain valid under the transported operations on \(F\),
since \(\psi\) is an isomorphism (equivalently, by (I2)–(I3) for the inner product).  
Given \(\alpha\in F\) and \(u,v\in V\) with \(u\pmb{\perp} v\), we have
\begin{multicols}{2}
\begin{enumerate}
  \item \(\alpha\cdot_{V}u \;\pmb{\perp}\; v\);
  \item \(u \;\pmb{\perp}\; \alpha\cdot_{V}v\).
\end{enumerate}
\end{multicols}
\noindent
Indeed, \(\langle \alpha\cdot_{V}u,\,v\rangle=\alpha\bullet\langle u,v\rangle=0\) and
\(\langle u,\,\alpha\cdot_{V}v\rangle=\overline{\alpha}\bullet\langle u,v\rangle=0\).

The next lemma records two standard consequences of the existence of an orthonormal basis.
Since, by the definition of strong regularity, \(Q(V)=V\), the notions of linear independence
and span agree with those in the classical vector–space setting (cf.\ \cite[Lemma 2.11]{MarquesMoore}).
Consequently, the proof of the following lemma is similar to the classical one transported through field isomorphism (cf.\ \cite[Proposition 6.9]{LA}).

\begin{lemm}\label{orthobasis}
Let \(((F,\pps,\bullet),\psi)\) be an \(F\)-like line, let \((V,+_{V},\cdot_{V})\) be a strongly
regular near–vector space over \((F,\cdot)\), and let \((V,\langle-,-\rangle)\) be an inner–product
near–space over \((F,\pps,\bullet)\). If \(\mathcal{U}=\{u_{1},\dots,u_{n}\}\) is an orthonormal
basis of \(V\), then:
\begin{enumerate}
  \item \(\mathcal{U}\) is linearly independent;
  \item if \(\mathcal{U}\) is finite, then for every \(v\in V\),
        \[
          v \;=\; {}^{V}\!\sum_{i=1}^{n}\, \langle v,u_{i}\rangle \bullet u_{i}.
        \]
\end{enumerate}
\end{lemm}

The following lemma is preparatory for the proof that every inner–product near–space admits
an orthonormal basis. Its proof mirrors the classical argument, transported along the field
isomorphism of the \(F\)-like line (cf.\ \cite[Theorem 6.12]{LA}).

\begin{lemm}[Bessel-type inequality]\label{BessIne}
Let \(((F,\pps,\bullet),\psi)\) be an \(F\)-like line, let \((V,+_{V},\cdot_{V})\) be a strongly
regular near–vector space over \((F,\cdot)\), let \((V,\langle-,-\rangle)\) be an inner–product
near–space over \((F,\pps,\bullet)\), and let \(\|\!-\!\|\) be the associated norm from
Lemma~\ref{normassociated}. If \(\mathcal{U}=\{u_{1},\dots,u_{n}\}\subset V\) is a finite
orthonormal set, then for all \(v\in V\), 
\[
\bm{\sum}_{i=1}^{n} \norm{\langle v,u_{i}\rangle}^{\bm 2}\;\pmb{\leq}\;\|v\|^{\bm 2}.
\]
Moreover, if
\[
v' \;:=\; v \;-_V\; {}^{V}\!\sum_{i=1}^{n}\langle v,u_{i}\rangle \bullet u_{i},
\]
then \(v'\) is orthogonal to every vector in \(\mathcal{U}\); that is, \(v'\perp u_{i}\) for
\(i=1,\dots,n\).
\end{lemm}
 As for classical inner product spaces, orthonormal bases exist; the proof is again the transported version of (cf. \cite[Proposition 6.9]{LA}).

\begin{prop}
    Let $((F,\pps,\bullet), \psi)$ be an $F$-like line, $((V,+_{V}),\cdot_{V})$ be a strongly regular near-vector space over $(F,\cdot)$ and $(V, \langle -,- \rangle$) be a finite dimensional inner product near-space over $(F,\pps,\bullet)$ and $\|-\|$ be the norm from Lemma \ref{normassociated}. Then, orthonormal bases exist in $(V, \langle -,- \rangle)$.
\end{prop}

In the following proposition, we shall denote the sum $u_{1} \pps \cdots \pps u_{n}$ by $\bm{\sum}_{i=1}^{n}u_{i}$.

\begin{lemm}\label{squares}
    Let $((F,\pps,\bullet), \psi)$ be an $F$-like line, $((V,+_{V}),\cdot_{V})$ be a strongly regular near-vector space over $(F,\cdot)$ and $(V, \langle -,- \rangle$) be an inner product near-space over $(F,\pps,\bullet)$ and $\|-\|$ be the norm from Lemma \ref{normassociated}. Let $u_{1},\cdots, u_{n} \in V$. If $\{ u_{1},\cdots, u_{n}\}$ is a set of orthogonal vectors, then 
   \begin{equation*}
   {}^{}\!\bm{\sum}_{i=1}^{n} \|u_{i}\|^{\bm{2}} = \left\|{}^{V}\!\sum_{i=1}^{n} u_{i}\right\|^{\bm{2}}. 
    \end{equation*}
\end{lemm}

\section{Duality and dual space of near-vector spaces}

In this section, we examine the duality theory for near-vector spaces and its connection with inner products. 
The main difficulty in the near-linear setting lies in proving that the dual space is itself a near-vector space, namely in showing that the quasi-kernel generates. 
We establish this result in the finite-dimensional case, where the dual can indeed be endowed with a near-vector space structure.

\begin{definlem}
Let $(F,+,\cdot)$ be a near-field and $V$ be a near-vector space over $F$. Consider
\[
V^{\circledast}:=\operatorname{Hom}(V,F),
\]
endowed with pointwise operations: 
\[
(\alpha f)(v):=\alpha\,f(v)\quad\text{and}\quad (f+g)(v):=f(v)+g(v)
\qquad(\alpha\in F,\; f,g\in V^{\circledast},\; v\in V),
\]
where the maps are \(F\)-linear \(V\to F\) with respect to the \(F\)-space structure on \(V\) and the near-field structure on \(F\). We call $V^{\circledast}$ the {\sf dual near-space of $V$ over $(F,+,\cdot)$}.
These operations make \(V^{\circledast}\) an \(F\)-space.
\end{definlem}

\begin{rem}
For a finite-dimensional near-vector space, the dual depends on the near-field structure chosen on the scalar group over which the near-vector space is defined.
In what follows, our notation does not distinguish between these possibilities.
\end{rem}

We now define the dual operator associated with a given set.

\begin{defin}
Let $(F,+,\cdot)$ be a near-field, and let $V$ be a near-vector space over $F$.  
Given a subset $\mathcal{B} = \{b_i \mid i \in I\} \subseteq V$, we define for each $j \in I$ a map
\[
b_j^{\circledast_{\mathcal{B}}} : \mathcal{B} \longrightarrow F, \qquad 
b_j^{\circledast_{\mathcal{B}}}(b_i) = \delta_{j,i} \quad \text{for all } b_i \in \mathcal{B}.
\]
The family 
\[
\mathcal{B}^{\circledast} := \{\, b_j^{\circledast_{\mathcal{B}}} \mid j \in I \,\}
\]
is called the \emph{dual family} associated with $\mathcal{B}$.
\end{defin}

For finite-dimensional \(V\), the next lemma shows
that every \(F\)-basis admits a corresponding dual basis in
\(V^{\circledast}\).

\begin{lemm}
\label{dualbasis}
Let $(F,+,\cdot)$ be a near-field, let $V$ be a near-vector space over $F$, and $\mathcal{B} = \{b_i \mid i \in I\} \subseteq V$. Then:
\begin{enumerate}
    \item if $\mathcal{B}$ is linearly independent of $V$, then $\mathcal{B}^{\circledast}$ is linearly independent of $V^{\circledast}$;
    \item if $\mathcal{B}$ is a finite spanning set of $V$, then $\mathcal{B}^{\circledast}$ spans $V^{\circledast}$;
    \item if $\mathcal{B}$ is a finite $F$-basis of $V$, then $\mathcal{B}^{\circledast}$ is a finite $F$-basis of $V^{\circledast}$.
\end{enumerate}
In particular, if $V$ is finite-dimensional, then $V$ and its dual $V^{\circledast}$ have the same minimal and maximal dimensions.
\end{lemm}

\begin{proof}
\begin{enumerate}
    \item 
Suppose $\sum_{j \in J}\left(\sum_{k=1}^{m}\beta_{j,k} \cdot b_j^{\circledast_{\mathcal{B}}} \right) = 0$ where $J \finsub I$, $m \in \mathbb{N}$ and $\beta_{j,k} \in F$ for all $j \in I$ and $k \in \{1,\cdots, m\}$. Then, for all $b_{i} \in \mathcal{B}$, we have 
    \begin{align*}
\sum_{j \in J}\left(\sum_{k=1}^{m}\beta_{j,k} \cdot b_j^{\circledast_{\mathcal{B}}} \right) (b_i) = \sum_{j \in J}\sum_{k=1}^{m}\beta_{j,k} \cdot b_j^{\circledast_{\mathcal{B}}} (b_i) = \sum_{k=1}^{m} \beta_{j,k} = 0. 
\end{align*}
This is equivalent to $\sum_{k=1}^{m}\beta_{j,k}\cdot b_j^{\circledast_{\mathcal{B}}} = 0$ for all $j \in \{1,\cdots,n\}$,
which proves linear independence.
\item Let $\psi \in V^{\circledast}$. Then for all $b_i \in \mathcal{B}$, we see that $\sum_{j\in I}\psi(b_{j})\cdot b_j^{\circledast_{\mathcal{B}}} (b_{i}) = \psi(b_i)$ and so $\psi = \sum_{j \in I}\psi(b_j)\cdot b_j^{\circledast_{\mathcal{B}}}$, which shows that $\{b_{j}^{\circledast_{\mathcal{B}}} \mid j \in I\}$ spans $V^\circledast$.
\item Suppose $\mathcal{B}$ is a finite $F$-basis for $V$. Then it follows by 1. and 2. that $\mathcal{B}^{\circledast}$ is a finite $F$-basis for $V^{\circledast}$.
  \end{enumerate}  
\end{proof}
Recall that an $F$-space $V$ is a near-vector space over $F$ if and only if it admits a scalar basis over $F$ (see \cite[Theorem~3.11]{MarquesMoore}). 
Consequently, the following lemma is an immediate corollary of Lemma~\ref{dualbasis}.
\begin{prop}
    Let $(F,+,\cdot)$ be a near-field and $V$ be a finite near-vector space over $F$. Then, $V^{\circledast}$ is a near-vector space over $F$.
\end{prop}

Each inner product near-space structure naturally gives rise to a corresponding dual space, as defined below.

\begin{definlem}
Let $((F,\pps,\bullet), \psi)$ be an $F$-like line, $((V,+_{V}),\cdot_{V})$ a strongly regular near-vector space over $(F,\cdot)$, and $(V, \langle -,- \rangle)$ an inner product near-space over $(F,\pps,\bullet)$. 
For each $u \in V$, define
\[
f_{u}: V \rightarrow F, \qquad f_{u}(v):= \langle v, u \rangle .
\]
We call $f_{u}$ a {\sf linear functional map} on $V$. 
The set of all linear functional maps,
\[
V^{\curlyvee}:=\{\, f_{u} \mid u \in V \,\},
\]
equipped with pointwise addition and scalar multiplication inherited from $F$, is called the {\sf dual near-vector space with the inner near-vector space $(V, \langle -,- \rangle)$ on $(((V,+_{V}),\cdot_{V}),(F,\cdot))$ over $(F,\pps,\bullet)$} and is denoted by $V^\curlyvee$. 
Furthermore, $V^{\curlyvee}$ is a subspace of $V^{\circledast}$.
\end{definlem}

The following theorem states how every finite dimensional near-vector space is naturally isormorphic to its double dual, as in the classical case. The proof mirrors the proof provided in \cite[Lemma 3.23]{LA}. 

\begin{theo}
Let $(F,+,\cdot)$ be a near-field and $V$ a finite-dimensional near-vector space over $F$. 
Define
\[
\eta : V \longrightarrow (V^{\circledast})^{\circledast}, 
\qquad 
(\eta(v))(\psi) := \psi(v) \quad \text{for all } v \in V,\ \psi \in V^{\circledast}.
\]
Then $\eta$ is an $F$-isomorphism of near-vector spaces.
\end{theo}

Finally, we define the notions of annihilator and orthogonal complement in the setting of near-vector spaces. 
These constructions determine important subspaces of the dual near-spaces introduced earlier.

\begin{definlem}

\begin{enumerate}
  \item Let $(F,+,\cdot)$ be a near-field and $V$ a finite dimensional near-vector space over $F$. 
  If $W$ is a subspace of $V$, the {\sf annihilator of $W$ in $V^{\circledast}$}, denoted $W^{\circ}$, is
  \[
    W^{\circ}:= \{ \psi \in V^{\circledast} \mid \psi(w) = 0 \text{ for all } w \in W \}.
  \]
  In particular, $W^{\circ}$ is a subspace of $V^{\circledast}$.

  \item When $(F,\pps,\bullet)$ is an $F$-like line, $((V,+_{V}),\cdot_{V})$ is a strongly regular near-vector space over $(F,\cdot)$, and $(V, \langle -,- \rangle)$ is an inner product near-space on $(((V,+_{V}),\cdot_{V}),(F,\cdot))$ over $(F,\pps,\bullet)$, then for any $C \subseteq V$ we define the {\sf orthogonal complement of $C$} to be the annihilator of $C$ in $V^{\curlyvee}$,
  \[
    C^{\pmb{\perp}} := \{ \varphi \in V^{\curlyvee} \mid \varphi(c) = 0 \text{ for all } c \in C \}.
  \]
  In particular, $C^{\pmb{\perp}}$ is a subspace of $V^{\curlyvee}$.
\end{enumerate}
\end{definlem}

Before describing bases of annihilators, we first establish a separation property for near-vector spaces: 
any vector lying outside a proper subspace can be separated from it by a functional in the dual that vanishes on the subspace but not on the given vector. 
This result plays the same foundational role as in classical linear algebra. The proof follows the classical vector-space argument (cf.\ \cite[Lemma~3.25]{LA}) transported to the $F$-like line.

\begin{lemm}
Let $(F,+,\cdot)$ be a near-field, $V$ a finite-dimensional near-vector space over $F$, and $W$ a proper subspace of $V$. 
If $v_{0} \in V \setminus W$, then there exists $\psi \in V^{\circledast}$ such that $\psi(v_0) \neq 0$ and $\psi \in W^{\circ}$.
\end{lemm}

The next lemma gives an explicit description of the annihilator of a subspace via the dual basis. 
It also prepares the ground for the dimension formula below, relating \(\dim W\), \(\dim W^{\circ}\), and \(\dim V\).
For the next lemma, see also \cite[Exercise 3.26]{LA}.

\begin{lemm}
\label{dualdim}
Let $(F,+,\cdot)$ be a near-field, $V$ a finite-dimensional near-vector space over $F$, and $W$ a proper subspace of $V$. 
Suppose $\mathcal{B} = \{b_i \mid i \in I\}$ is a finite $F$-basis of $V$ and $\mathcal{B}_1 = \{ b_j \mid j \in J\} \subset \mathcal{B}$ is a finite $F$-basis of $W$ with $J\subseteq I$. 
Then
\[
\mathcal{B}_2^{\circledast} := \{\, b_k^{\circledast_{\mathcal{B}}}\mid k \in I \setminus J \,\}
\]
is a finite $F$-basis of $W^{\circ}$.
\end{lemm}

For the following theorem, see \cite[Exercise 3.27]{LA} for reference of the classical case.
\begin{theo}[Dimension formula for $W^{\circ}$]
Let $(F,+,\cdot)$ be a near-field and $V$ a finite-dimensional near-vector space over $F$. 
If $W$ is a subspace of $V$, then
\[
\operatorname{dim}(W) + \operatorname{dim}(W^{\circ}) = \operatorname{dim}(V).
\]
\end{theo}
\section{A definition of an angle }
\vspace{0cm}
 A finer notion than orthogonality is that of an \emph{angle}. 
In the classical setting, angles are defined via the Cauchy--Schwarz inequality; for inner near–vector spaces, an analogue was established in Theorem~\ref{cauchy}. 
In this section, we extend the notion of angle to the near–vector space context and record its basic properties.

Let $\pps$ and $\bullet$ be binary operations on $F$, so that $((F,\pps,\bullet),\psi)$ is an $F$-like line. 
Let $((V,+_{V}),\cdot_{V})$ be a strongly regular near-vector space over $(F,\cdot)$, let $(V,\langle-,-\rangle)$ be an inner product near-space over $(F,\pps,\bullet)$, and let $\|\!-\!\|$ be the norm from Lemma~\ref{normassociated}. 
By Theorem~\ref{cauchy}, for all nonzero $u,v\in V$ we have
\[
\mathbf{Re}\!\big(\langle u,v\rangle\big)\;\bullet\;\|u\|^{\bm{-1}}\;\bullet\;\|v\|^{\bm{-1}} \in [-1,1],
\]
which permits the following definition.

\begin{defin}
Let $\pps$ and $\bullet$ be binary operations on $F$, so that $((F,\pps,\bullet),\psi)$ is an $F$-like line. 
Let $((V,+_{V}),\cdot_{V})$ be a strongly regular near-vector space over $(F,\cdot)$, let $(V,\langle-,-\rangle)$ be an inner product near-space over $(F,\pps,\bullet)$, and let $\|\!-\!\|$ be the norm from Lemma~\ref{normassociated}. 
For nonzero $u,v\in V$, we define the \emph{angle} of the pair $(u,v)$ by
\[
\angle(u,v)\;:=\;\arccos\!\Big(\,\mathbf{Re}\!\big(\langle u,v\rangle\big)\;\bullet\;\|u\|^{\bm{-1}}\;\bullet\;\|v\|^{\bm{-1}}\,\Big).
\]
By definition, $\angle(u,v)\in[0,\pi]$.
\end{defin}

We now record several properties of the angle that closely parallel the classical setting. 
The proofs proceed analogously, adapted to the structure of the $F$-like line.

\begin{properties}
Let $\pps$ and $\bullet$ be binary operations on $F$, so that $((F,\pps,\bullet),\psi)$ is an $F$-like line. 
Let $((V,+_{V}),\cdot_{V})$ be a strongly regular near-vector space over $(F,\cdot)$, let $(V,\langle-,-\rangle)$ be an inner product near-space over $(F,\pps,\bullet)$, and let $\|-\|$ be the norm from Lemma~\ref{normassociated}. 
For $u,v \in V$ with $u,v \neq 0$, the following hold:
\begin{enumerate}
  \item $u \pmb{\perp} v$ if and only if $\angle(u,v) = \dfrac{\pi}{2}$;
  \item If $F=\mathbb{R}$, then for all $\alpha,\beta \in \mathbb{R}^{*}$ with $\alpha$ and $\beta$ of the same sign, we have 
        $\angle(u,v) = \angle(\alpha \cdot_{V} u,\, \beta \cdot_{V} v)$;
  \item $\angle(u,v) = \angle(v,u)$;
  \item $\displaystyle 
        \|\,u -_{V} v\,\|^{\bm{2}} 
        = \|u\|^{\bm{2}} \mms \bm{2} \bullet \|u\| \bullet \|v\| \bullet \cos\!\big(\angle(u,v)\big) \pps \|v\|^{\bm{2}}.$
\end{enumerate}
\end{properties}

\section{Example of an inner product on a multiplicative near-vector space}

In this section, we construct a broad class of inner product near-spaces. 
These examples illustrate how certain classical norms, those that do not give rise to Hilbert spaces, can naturally be recovered within the near-vector space framework.

\begin{assumption}[Notations]\label{ass1}
Throughout this section, we assume the following:
\begin{itemize}
    \item[\textbullet] $F \in \{\mathbb{R}, \mathbb{C}\}$;
    
    \item[\textbullet] $V$ is a vector space over $(F, +, \cdot)$, and for each $i \in I$, the subspace $W_i \subseteq V$ is equipped with a basis $\mathcal{B}_i$ over $F$;
    
    \item[\textbullet] For each $i \in I$, define $\theta_i \colon W_i \to F^{\mathcal{B}_i}$ by mapping each basis vector $b \in \mathcal{B}_i$ to $(\delta_{b,c})_{c \in \mathcal{B}_i}$ and extending linearly. Then $\theta_i$ is a vector space isomorphism over $(F,+,\cdot)$, and we define $\boldsymbol{\theta} = (\theta_i)_{i \in I}$;
    
    \item[\textbullet] $(F^{\mathcal{B}_i}, \langle -,- \rangle_i)_{i \in I}$ is a family of inner product near-vector spaces over $(F, +_{\lambda \circ \varphi^{-1}}, \cdot)$;
    
    \item[\textbullet] For each $i \in I$, let $\|{-}\|_i$ denote the norm associated to the inner product $\langle -,- \rangle_i$, as defined in Definition-Lemma~\ref{normassociated};
    
    \item[\textbullet] Let $\mathcal{B} := \bigcup_{i \in I} \mathcal{B}_i$;
    
    \item[\textbullet] For each $b \in \mathcal{B}$, let $\rho_b$ be a multiplicative automorphism of $F$;
    
    \item[\textbullet] The maps $\varphi$ and $\lambda$ are multiplicative automorphisms of $F$ that preserve complex conjugation (i.e., when $F = \mathbb{C}$, we have $\varphi(\overline{\alpha}) = \overline{\varphi(\alpha)}$, and similarly for $\lambda$);
    
    \item[\textbullet] The maps $\varphi$ and $\lambda$ are order-preserving over $\mathbb{R}^{+}$: for all $\alpha, \beta \in \mathbb{R}^{+}$, if $\alpha \leq \beta$ then $\varphi(\alpha) \leq \varphi(\beta)$ and $\lambda(\alpha) \leq \lambda(\beta)$;
    
    \item[\textbullet] For each $i \in I$, define $\boldsymbol{\rho}_i := (\rho_b)_{b \in \mathcal{B}_i}$, $\boldsymbol{\sigma}_i := (\varphi \circ \rho_b^{-1})_{b \in \mathcal{B}_i}$, and $\widetilde{\boldsymbol{\sigma}_i} := (\varphi \circ \rho_b^{-1})_{b \in \mathcal{B}_i}\theta_i$;
    
    \item[\textbullet] Define $\boldsymbol{\rho} := (\boldsymbol{\rho}_i)_{i \in I}$, $\boldsymbol{\sigma} := (\boldsymbol{\sigma}_i)_{i \in I}$, and $\widetilde{\boldsymbol{\sigma}} := (\widetilde{\boldsymbol{\sigma}_i})_{i \in I}$;
    
    \item[\textbullet] To simplify notation, for each $i \in I$, define binary operations $+_{\widetilde{\boldsymbol{\sigma}_i}}$ and $\cdot_{\widetilde{\boldsymbol{\rho}_i}}$ on $W_i$ by:
    \[
    u_i +_{\widetilde{\boldsymbol{\sigma}_i}} v_i := \theta_i^{-1}\big(\theta_i(u_i) +_{\boldsymbol{\sigma}_i} \theta_i(v_i)\big), \quad
    \alpha \cdot_{\widetilde{\boldsymbol{\rho}_i}} u_i := \theta_i^{-1}\big(\alpha \cdot_{\boldsymbol{\rho}_i} \theta_i(u_i)\big),
    \]
    for all $u_i, v_i \in W_i$ and $\alpha \in F$.
    
    \item[\textbullet] Similarly, define binary operations $+_{\widetilde{\boldsymbol{\sigma}}}$ and $\cdot_{\widetilde{\boldsymbol{\rho}}}$ on $W := \bigoplus_{i \in I} W_i$ by:
    \[
    u +_{\widetilde{\boldsymbol{\sigma}}} v := \boldsymbol{\theta}^{-1}\big(\boldsymbol{\theta}(u) +_{\boldsymbol{\sigma}} \boldsymbol{\theta}(v)\big), \quad
    \alpha \cdot_{\widetilde{\boldsymbol{\rho}}} u := \boldsymbol{\theta}^{-1}\big(\alpha \cdot_{\boldsymbol{\rho}} \boldsymbol{\theta}(u)\big),
    \]
    for all $u, v \in W$ and $\alpha \in F$;
    
    \item[\textbullet] Define $\eta_i := (\theta_i, \varphi, \boldsymbol{\rho}_i)$, for each $i \in I$, and $\boldsymbol{\eta} := (\boldsymbol{\theta}, \varphi, \boldsymbol{\rho})$;
    
    \item[\textbullet] Then, $W_i^{\eta_i}$ denotes the near-vector space $\big(W_i, +_{\widetilde{\boldsymbol{\sigma}_i}}, \cdot_{\widetilde{\boldsymbol{\rho}_i}}\big)$ over $(F, \cdot)$, and $W^{\boldsymbol{\eta}}$ denotes the near-vector space $\big(W, +_{\widetilde{\boldsymbol{\sigma}}}, \cdot_{\widetilde{\boldsymbol{\rho}}}\big)$ over $(F, \cdot)$.
\end{itemize}
\end{assumption}

\begin{rem}\label{multautpreserve}
If $\kappa$ is an order-preserving multiplicative automorphism over $\mathbb{R}^{+}$, then the following hold:
\begin{enumerate}
    \item $\kappa(0) = 0$;
    \item $\kappa(\mathbb{R}^{+}) \subseteq \mathbb{R}^{+}$ and $\kappa^{-1}(\mathbb{R}^{+}) \subseteq \mathbb{R}^{+}$;
    \item $\kappa(\mathbb{R}^{+} \setminus \{0\}) \subseteq \mathbb{R}^{+} \setminus \{0\}$ and $\kappa^{-1}(\mathbb{R}^{+} \setminus \{0\}) \subseteq \mathbb{R}^{+} \setminus \{0\}$.
\end{enumerate}
\end{rem}
We now establish an isomorphism between our constructed near-vector spaces and multiplicative near-vector spaces.

\begin{lemm}\label{innerprod}
With the notation from Assumption~\ref{ass1}, we have
\[
W_i^{\eta_i} \simeq F^{\boldsymbol{\sigma}_i, \boldsymbol{\rho}_i}
\]
via the isomorphism $(\theta_i, \operatorname{id})$. Let $\boldsymbol{\iota}_i \colon F^{\boldsymbol{\sigma}_i, \boldsymbol{\rho}_i} \to F^{\boldsymbol{\sigma}, \boldsymbol{\rho}}$ denote the canonical embedding, and define
\[
\boldsymbol{\theta} := \bigoplus_{i \in I} \boldsymbol{\iota}_i \circ \theta_i.
\]
Then we have:
\[
W^{\boldsymbol{\eta}} \simeq F^{\boldsymbol{\sigma}, \boldsymbol{\rho}}
\]
via the isomorphism $(\boldsymbol{\theta}, \operatorname{id})$.

In particular, for each $i \in I$, both $W_i^{\eta_i}$ and $W^{\boldsymbol{\eta}}$ are multiplicative, strongly regular near-vector spaces.
\end{lemm}

\begin{proof}
Fix $i \in I$. Define the map
\[
\theta_i \colon W_i^{\eta_i} \to F^{\boldsymbol{\sigma}_i, \boldsymbol{\rho}_i}
\]
as before: each basis element $b \in \mathcal{B}_i$ is mapped to $(\delta_{b,c})_{c \in \mathcal{B}_i}$, extended linearly. By definition, $\theta_i$ is a bijective linear isomorphism.

Let $\alpha \in F$ and $u_i, v_i \in W_i^{\eta_i}$. Then:
\[
\theta_i(u_i +_{\widetilde{\boldsymbol{\sigma}_i}} v_i)
= \theta_i\left(\theta_i^{-1}\left(\theta_i(u_i) +_{\boldsymbol{\sigma}_i} \theta_i(v_i)\right)\right)
= \theta_i(u_i) +_{\boldsymbol{\sigma}_i} \theta_i(v_i),
\]
and
\[
\theta_i(\alpha \cdot_{\widetilde{\boldsymbol{\rho}_i}} u_i)
= \theta_i\left(\theta_i^{-1}\left(\alpha \cdot_{\boldsymbol{\rho}_i} \theta_i(u_i)\right)\right)
= \alpha \cdot_{\boldsymbol{\rho}_i} \theta_i(u_i).
\]
Thus, $\theta_i$ preserves the operations and defines an isomorphism:
\[
W_i^{\eta_i} \simeq F^{\boldsymbol{\sigma}_i, \boldsymbol{\rho}_i}.
\]

Now, let $\boldsymbol{\iota}_i \colon F^{\boldsymbol{\sigma}_i, \boldsymbol{\rho}_i} \to F^{\boldsymbol{\sigma}, \boldsymbol{\rho}}$ be the canonical embedding. Then, for each $i \in I$, we have:
\[
W_i^{\eta_i} \simeq \boldsymbol{\iota}_i(F^{\boldsymbol{\sigma}_i, \boldsymbol{\rho}_i}) \subseteq F^{\boldsymbol{\sigma}, \boldsymbol{\rho}}.
\]

Since $F^{\boldsymbol{\sigma}, \boldsymbol{\rho}} = \bigoplus_{i \in I} \boldsymbol{\iota}_i(F^{\boldsymbol{\sigma}_i, \boldsymbol{\rho}_i})$ and $W^{\boldsymbol{\eta}} = \bigoplus_{i \in I} W_i^{\eta_i}$, we define the map
\[
\boldsymbol{\theta} \colon W^{\boldsymbol{\eta}} \to F^{\boldsymbol{\sigma}, \boldsymbol{\rho}}, \quad
\boldsymbol{\theta}(u) := \bigoplus_{i \in I} \boldsymbol{\iota}_i(\theta_i(u_i)),
\]
which is an isomorphism. Hence,
\[
W^{\boldsymbol{\eta}} \simeq F^{\boldsymbol{\sigma}, \boldsymbol{\rho}}.
\]

By \cite[Theorem~3.16]{LSD2025}, we know that both $W_i^{\eta_i}$ and $W^{\boldsymbol{\eta}}$ are multiplicative near-vector spaces.

  We note that, for all $i \in I$, we have
\[
+_{\boldsymbol{\sigma}_i \circ \boldsymbol{\rho}_i} = +_\varphi, \quad \text{and} \quad +_{\boldsymbol{\sigma} \circ \boldsymbol{\rho}} = +_\varphi.
\]
By \cite[Definition–Lemma~3.11]{LSD2025}, we therefore have $+_{v}=+_{w}$, for all $v,w\in Q(V)^*$.
It then follows by Remark \ref{strongly} that $W_i^{\eta_i}$ and $W^{\boldsymbol{\eta}}$ are strongly regular.
\end{proof}

\begin{definlem} \label{inner}
 We adopt the assumptions and notations introduced in Assumptions \ref{ass1}. We define an {\sf inner product on \((W^{\boldsymbol{\eta}}, (F, \cdot))\) over \((F, +_{\lambda}, \cdot)
 \)}, denoted by $\langle -, - \rangle_{\mathcal{I}}$, where 
 $$\mathcal{I} = (\boldsymbol{\theta}, \boldsymbol{\sigma}, \boldsymbol{\rho}, \varphi, \lambda,W,(W_i,\langle-,-\rangle_{i})_{i \in I}),$$ to be the map  
\[
\langle -, - \rangle_{\mathcal{I}}\colon W \times W \longrightarrow F,
\]  
defined as follows: for any \(u, v \in W\), write  
\[
u = \sum_{i \in I} u_i, \quad v = \sum_{i \in I} v_i, \quad 
(u_i, v_i \in W_i \text{ for all } i \in I)
\]  
and we set  
\[
\langle u, v \rangle_{\mathcal{I}} = {}^{\lambda}\! \sum_{i \in I} \varphi^{-1} \left( \left\langle \widetilde{\boldsymbol{\sigma}_i}(u_i), \widetilde{\boldsymbol{\sigma}_i}(v_i) \right\rangle_i \right).
\]  

\end{definlem}
\begin{proof} Since $\lambda$ preserves complex conjugation, it follows immediately that $\lambda^{-1}$ does as well, and consequently so does $+_{\lambda}$.
We also use the fact that $\lambda $ and $\varphi$ is order preserving, that is, $\lambda(\mathbb{R}^{+}) = \mathbb{R}^{+}$ and $\varphi(\mathbb{R}^{+}) = \mathbb{R}^{+}$. Since $\varphi$ is a multiplicative automorphism and is order preserving, we know that $\varphi^{-1}$ is also order preserving, and  $\varphi^{-1}(0) = 0$. Next, we need to verify assumptions $I1.-I5.$ from Definition \ref{innerproductnvs} to prove that $\langle -,-\rangle_{\mathcal{I}}$ is indeed an inner product on \((W^{\boldsymbol{\eta}}, (F, \cdot))\) over \((((F, +_{\varphi}), \cdot),(F, \cdot))\). Let $\alpha \in F$ and $u,v,w \in W$ where $u = \sum_{i \in I}u_{i}, v = \sum_{i \in I}v_{i}$ and $w = \sum_{i \in I}w_{i}$ with $u_i, v_i, w_i\in W_i$ for all $i\in I$.     
\begin{enumerate}

    \item[\textbf{(I1)}] Since, by assumption, \( \varphi \) preserves complex conjugation, we obtain the following sequence of equalities:
    \begin{align*}
        \overline{\langle v,u\rangle_{\mathcal{I}}} 
        &= {}^{\lambda} \sum_{i \in I} \overline{\varphi^{-1}(\langle \widetilde{\boldsymbol{\sigma}_i}(v_{i}), \widetilde{\boldsymbol{\sigma}_i}(u_{i})\rangle_{i})} \\
        &= {}^{\lambda} \sum_{i \in I} \varphi^{-1}(\overline{\langle \widetilde{\boldsymbol{\sigma}_i}(v_{i}), \widetilde{\boldsymbol{\sigma}_i}(u_{i})\rangle_{i}}) \\
        &= {}^{\lambda} \sum_{i \in I} \varphi^{-1}(\langle \widetilde{\boldsymbol{\sigma}_i}(u_{i}), \widetilde{\boldsymbol{\sigma}_i}(v_{i})\rangle_{i}) \\
        &= \langle u, v \rangle_{\mathcal{I}}.
    \end{align*}

    \item[\textbf{(I2)}]
    \begin{align*}
        \langle \alpha \cdot_{ \boldsymbol{\widetilde{\rho}}}u,w\rangle_{\mathcal{I}} 
        &= {}^{\lambda} \sum_{i \in I}  \varphi^{-1} (\langle \boldsymbol{\widetilde{\sigma}}_{i}(\theta_{i}^{-1}(\alpha\cdot_{\boldsymbol{\rho}_{i}}\theta_{i}(u_{i}))),\boldsymbol{\widetilde{\sigma}}_{i}(v_{i})\rangle_{i}) \\
        &= {}^{\lambda} \sum_{i \in I} \varphi^{-1} (\langle \boldsymbol{\sigma}_{i}(\alpha\cdot_{\boldsymbol{\rho}_{i}}\theta_{i}(u_{i})),\boldsymbol{\widetilde{\sigma}}_{i}(v_{i})\rangle_{i}) \\
        &= {}^{\lambda} \sum_{i \in I} \varphi^{-1} (\langle\varphi(\alpha)\boldsymbol{\widetilde{\sigma}}_{i}(u_{i}),\boldsymbol{\widetilde{\sigma}}_{i}(v_{i})\rangle_{i}) \\
        &= {}^{\lambda} \sum_{i \in I} \varphi^{-1}(\varphi(\alpha)) \cdot \varphi^{-1}(\langle \boldsymbol{\widetilde{\sigma}}_{i}(u_{i}), \boldsymbol{\widetilde{\sigma}}_{i}(v_{i})\rangle_{i}) \\
        &= \alpha \cdot {}^{\lambda} \sum_{i \in I} \varphi^{-1}(\langle \boldsymbol{\widetilde{\sigma}}_{i}(u_{i}), \boldsymbol{\widetilde{\sigma}}_{i}(v_{i})\rangle_{i}) \\
        &= \alpha \cdot \langle u,v\rangle_{\mathcal{I}}.
    \end{align*}

    \item[\textbf{(I3)}]
    \begin{align*}
        \langle u +_{\widetilde{\boldsymbol{\sigma}}}v,w\rangle_{\mathcal{I}} 
        &= {}^{\lambda}\! \sum_{i \in I} \varphi^{-1} (\langle \widetilde{\boldsymbol{\sigma}_i}(u_i +_{\widetilde{\boldsymbol{\sigma}_i}}v_i) , \widetilde{\boldsymbol{\sigma}_i}(w_i)\rangle_i) \\
        &= {}^{\lambda}\! \sum_{i \in I} \varphi^{-1} (\langle \widetilde{\boldsymbol{\sigma}_i}(u_i) + {\widetilde{\boldsymbol{\sigma}_i}} (v_i) , \widetilde{\boldsymbol{\sigma}_i}(w_i)\rangle_i) \\
        &= {}^{\lambda}\! \sum_{i \in I} \varphi^{-1} ( \langle \widetilde{\boldsymbol{\sigma}_i}(u_i), \widetilde{\boldsymbol{\sigma}_i}(w_i)\rangle_i +_{\lambda \circ \varphi^{-1}} \langle {\widetilde{\boldsymbol{\sigma}_i}} (v_i), \widetilde{\boldsymbol{\sigma}_i}(w_i)\rangle_i ) \\
        &= {}^{\lambda}\! \sum_{i \in I} \left( \varphi^{-1}(\langle \widetilde{\boldsymbol{\sigma}_i}(u_i), \widetilde{\boldsymbol{\sigma}_i}(w_i)\rangle_i) +_{\lambda} \varphi^{-1}(\langle {\widetilde{\boldsymbol{\sigma}_i}} (v_i), \widetilde{\boldsymbol{\sigma}_i}(w_i)\rangle_i) \right) \\
        &= {}^{\lambda}\! \sum_{i \in I} \varphi^{-1} ( \langle \widetilde{\boldsymbol{\sigma}_i}(u_i), \widetilde{\boldsymbol{\sigma}_i}(w_i)\rangle_i) +_{\lambda} {}^{\lambda}\! \sum_{i \in I} \varphi^{-1} ( \langle {\widetilde{\boldsymbol{\sigma}_i}} (v_i), \widetilde{\boldsymbol{\sigma}_i}(w_i)\rangle_i) \\
        &= \langle u,w\rangle_{\mathcal{I}}+_{\lambda} \langle v,w \rangle_{\mathcal{I}}.
    \end{align*}

    \item[\textbf{(I4)}] By the properties of \( \langle -, - \rangle_i \), we know that \( \langle \widetilde{\boldsymbol{\sigma}_i}(u_{i}), \widetilde{\boldsymbol{\sigma}_i}(u_{i}) \rangle_i \in \mathbb{R}^+ \) for all \( i \in I \). By Remark~\ref{multautpreserve} 2., we conclude that
    \[
        \langle u,v \rangle_{\mathcal{I}} = {}^{\lambda} \sum_{i \in I} \varphi^{-1}(\langle \widetilde{\boldsymbol{\sigma}_i}(u_{i}), \widetilde{\boldsymbol{\sigma}_i}(v_{i})\rangle_{i}) \in \mathbb{R}^+.
    \]

    \item[\textbf{(I5)}] Suppose \( u = 0 \). Then \( u_i = 0 \) for all \( i \in I \). Thus, by the properties of \( \langle -,- \rangle_i \) and Remark~\ref{multautpreserve} 1., we have
    \[
        {}^{\lambda}\! \sum_{i \in I} \varphi^{-1} \left( \left\langle \widetilde{\boldsymbol{\sigma}_i}(u_i), \widetilde{\boldsymbol{\sigma}_i}(u_i) \right\rangle_i \right)
        = {}^{\lambda}\! \sum_{i \in I} \varphi^{-1} (0) = 0.
    \]
    Conversely, suppose \( u \neq 0 \). Since \( u = \sum_{i \in I} u_i \), there exists \( i_0 \in I \) such that \( u_{i_0} \neq 0 \). Then by (I4), we know that
    \[
        \left\langle \widetilde{\boldsymbol{\sigma}_{i_0}}(u_{i_0}), \widetilde{\boldsymbol{\sigma}_{i_0}}(u_{i_0}) \right\rangle_{i_0} \in \mathbb{R}^+ \setminus \{0\} \text{ and } \left\langle \widetilde{\boldsymbol{\sigma}_{i}}(u_{i}), \widetilde{\boldsymbol{\sigma}_{i}}(u_{i}) \right\rangle_{i} \in \mathbb{R}^+ \text{ for all } i \in I \setminus \{i_0\}.\] It then follows by Remark~\ref{multautpreserve} 3. that
    \[
        {}^{\lambda}\! \sum_{i \in I} \varphi^{-1} \left( \left\langle \widetilde{\boldsymbol{\sigma}_i}(u_i), \widetilde{\boldsymbol{\sigma}_i}(u_i) \right\rangle_i \right) \in \mathbb{R}^+ \setminus \{0\}.
    \]

\end{enumerate}

\end{proof}


\section{Recovering norms with inner products on strongly regular near-vector spaces}
\subsection{Preliminaries}
We recall the multiplicative automorphisms of \(F\in\{\mathbb R,\mathbb C\}\) (cf.~\cite[§2]{BM2024}).
Every \(z\in\mathbb C\) admits a polar decomposition \(z=rs\) with \(r\in\mathbb R^+ \) and
\(s\in\mathbb S:=\{w\in\mathbb C:\,|w|=1\}\). For \(\alpha\in\mathbb C\setminus i\mathbb R,\) define
\[
\begin{array}{rccc}
\epsilon_\alpha: & \mathbb C & \longrightarrow & \mathbb C\\[0.2em]
& z=rs\neq 0 & \longmapsto & r^\alpha s\\[0.2em]
& 0 & \longmapsto &  0
\end{array}
\qquad\text{and}\qquad
\begin{array}{rccc}
\overline{\epsilon_\alpha}: & \mathbb C & \longrightarrow & \mathbb C\\[0.2em]
& z=rs\neq0 & \longmapsto & r^\alpha \,\overline{s}\\[0.2em]
& 0 & \longmapsto &  0
\end{array}
\]
The maps \(\epsilon_\alpha\) and \(\overline{\epsilon_\alpha}\) are precisely all
continuous multiplicative automorphisms of \((\mathbb C^*,\cdot)\). Moreover, $\epsilon_{\alpha}^{-1} = \epsilon_{\frac{1-i \operatorname{Im}(\alpha)}{\operatorname{Re}(\alpha)}}$ and $\overline{\epsilon_{\alpha}}^{-1} = \overline{\epsilon_{\frac{1+i \operatorname{Im}(\alpha)}{\operatorname{Re}(\alpha)}}}$.

For \(\alpha\in\mathbb R\setminus\{0\}\), restricting to \(\mathbb R\), we continue to write \(\epsilon_\alpha\) for the corestriction \(\epsilon_\alpha:\mathbb R\to\mathbb R\).
In explicit form,
\[
\epsilon_\alpha:\mathbb R\longrightarrow\mathbb R,\qquad
x\longmapsto \operatorname{sgn}(x)\,|x|^{\alpha}
=
\begin{cases}
x^{\alpha}, & x\ge 0,\\[0.3em]
0 , & x=0 ,\\[0.3em]
-(-x)^{\alpha}, & x<0.
\end{cases}
\]
These \(\epsilon_\alpha\) are exactly all continuous multiplicative automorphisms of \((\mathbb R^*,\cdot)\).
We denote by $+_\alpha$ the addition associated with the continuous multiplicative automorphism $\epsilon_\alpha$ where $\alpha\in F^{*}$.

Let \(F\in\{\mathbb R,\mathbb C\}\) and \(I\) be an index set.
A sequence \((a_i)_{i\in\mathbb N}\) with entries in \(F\) is canonically identified with the function
\[
a:\mathbb N\to F,\qquad i\longmapsto a(i)=a_i.
\]
We therefore freely refer to \(a:\mathbb N\to F\) as a {\sf sequence} in \(F\). This functional viewpoint
is especially convenient for discussing sequences of sequences.

For the remainder of this paper, instead of writing  $+_{\epsilon_{p}}$ and $\cdot_{\epsilon_p}$, we simply write $+_{p}$ and $\cdot_{p}$.

Let \(+_\alpha\) be a (fixed) addition law on \(F\) induced by a chosen multiplicative automorphism (e.g.\ via \(\epsilon_\alpha\)).
For a sequence \((a_i)_{i\in\mathbb N},\) we write
\[
{}^{\alpha}\!\sum_{i=1}^{\infty} a_i
\;:=\;
\lim_{n\to\infty}\;{}^{\alpha}\!\sum_{i=1}^{n} a_i,
\]
whenever the limit exists (and leave it undefined otherwise).

In the same spirit, one defines the {\sf Riemann sums associated with \(+_\alpha\)} for a function
\(f:[a,b]\to F\) by replacing ordinary finite sums with \(+_\alpha\), and the
{\sf integral with respect to \(+_\alpha\)} as the limit of such Riemann sums:
\[
{}^{^{^{\alpha}}}\!\!\!\!\!\int_a^b f(x)\,dx:=\epsilon_\alpha^{-1} \left( \int_a^b \epsilon_\alpha (f(x))\,dx\right),
\]
whenever the limit exists.

\subsection{Recovering finite norms}
\label{Lpnorms}
For the remainder of this paper, we denote the set $\{1,\cdots, n \}$ by $[n]$.
\paragraph{Recovering the \texorpdfstring{$\ell^p$}{lp} norms.}
Let \(I\in\bigl\{[n],\mathbb N\bigr\}\) and \(p\in \mathbb{R}^*\).
Consider \(W:=F^{I}\), the direct sum of its coordinate copies of \(F\).
We regard each copy of \(F\) as an inner product near-space with the sesquilinear form
\((\alpha,\beta)\mapsto\alpha\,\overline{\beta}\). Fix \(\lambda=\epsilon_{p}\), \(\sigma_i =\epsilon_{p/2}\) and \(\rho_i^{-1} =\epsilon_{1/2}\), for all $i\in I$; take
all remaining structural maps in Assumptions~\ref{ass1} to be identities. We denote the strongly regular near-vector space defined as such $W_{\ell^p}(I)$.

For \(u=(u_i)_{i\in I}\) and \(v=(v_i)_{i\in I}\) in \(F^{I}\), define
\[
\langle u,v\rangle_{\ell^p} \;:=\; \epsilon_{p}^{-1} \left( \!\sum_{i\in I}\,\epsilon_{p/2} \!\bigl(u_i\,\overline{v_i}\bigr) \right).
\]
Then, $\langle u,v \rangle_{\ell_p}$ is an inner product on $(W_{\ell_{p}}(I), +_{p/2},\cdot_{2})$ over $(F,+_{p},\cdot)$.
We can define the associated norm by
\[
\|u\|_{\ell^p} \;=\; \Bigl(\sum_{i\in I} |u_i|^p\Bigr)^{1/p}.
\]
When \(I=\mathbb N\) and $p\in [1, \infty)$, we obtain the standard $\ell^p(F)$ space
\[
\ell^p(F) \;=\; \bigl\{\,a:\mathbb N\to F \;\big|\; \|a\|_{\ell^p}<\infty\,\bigr\},
\]
with a norm taking values in the \(F\)-like line \((F, +_p, \cdot)\), 
and equipped with an inner product defined over this \(F\)-like line. We recall that, by the Jordan--von Neumann theorem \cite{JordanNeumann}, for \(p\neq 2\) the \(\ell^{p}\)-norm is not induced by any inner product; in particular, \(\ell^{p}\) is not a Hilbert space. 
Nevertheless, by allowing a more flexible underlying structure (via the \(F\)-line \((F,+_{p},\cdot)\)), we have shown that these norms can be modeled within a Hilbert framework.

\paragraph{A note about \texorpdfstring{$\ell^{p,q}$}{lpq} norms.}
Let \(A=(a_{i,j})\in F^{m\times n}\). The mixed \((p,q)\)-norm can be understood as a two–stage aggregation.
First, for each column \(j\), take the \(\ell_p\)-length in \(W_{\ell_p}([m])\),
\[
c_j \;:=\; \Big(\sum_{i\in [m]}|a_{i,j}|^{p}\Big)^{1/p}\ \in\ (F,+_{p},\cdot),
\]
which produces the vector \((c_1,\dots,c_n)\in (F,+_{p},\cdot)^{\,n}\). Equivalently, this is the column-wise map
\[
W_{\ell_p}([m])^{\,n}\longrightarrow (F,+_{p},\cdot)^{\,n},\qquad
(a_{i,1},\dots,a_{i,n})_{i \in [m]}\longmapsto (c_1,\dots,c_n),
\]
interpretable as a generalized “inner-product–type” evaluation on \(W_{\ell_p}([m])^n\) with values in a product of \(F\)-like lines.

Second, aggregate across columns by taking the \(\ell_q\)-length in \(W_{\ell_q}([n])\),
\[
\|A\|_{\ell^{p,q}}
  \;=\;
  \Big(\sum_{j \in [n]} c_j^{\,q}\Big)^{1/q}
  \;=\;
  \Bigg(\sum_{j \in [n]}\Big(\sum_{i \in [m]}|a_{i,j}|^{p}\Big)^{\!q/p}\Bigg)^{\!1/q},
\]
which lands in the \(F\)-like line \((F,+_{q},\cdot)\).

In short, we apply \(p\)-aggregation within columns (into \((F,+_{p},\cdot)\)) and then \(q\)-aggregation across columns (into \((F,+_{q},\cdot)\)), yielding the standard mixed \((p,q)\)-norm. This may be viewed as the composition of the two aggregation steps described above. In conclusion, this process yields an aggregated inner–product–type composition. For any 
\(A=(a_{i,j})\) and \(B=(b_{i,j})\) in \(F^{m\times n}\), define
\[
\langle A,B\rangle_{\ell^{p,q}}
\;:=\;
\epsilon_{1/q}\!\left(
  \sum_{j \in [n]}
  \epsilon_{\,\frac{q}{2p}}\!\left(
    \sum_{i \in [m]}
    \epsilon_{\,\frac{p}{2}}\!\bigl(a_{i,j}\,\overline{b_{i,j}}\bigr)
  \right)
\right).
\]

\medskip

\
\paragraph{Recovering the \texorpdfstring{$\mathcal{L}^p$}{Lp} norms.}
Let \(I=[a,b]\) and let \(\mathcal{F}_{\mathrm{int}}(I)\) denote the space of integrable functions \(f:I\to F\), with \(F\in\{\mathbb R,\mathbb C\}\).
For each \(n\in\mathbb N\), consider the uniform partition
\[
\mathcal{P}_n:\quad a=x_0^{(n)}<x_1^{(n)}<\cdots<x_n^{(n)}=b,\qquad
x_k^{(n)}=a+\tfrac{k}{n}(b-a),
\]
and let \(S_n\) be the finite-dimensional subspace spanned by the step (indicator) functions
\(\chi_{[x_{k-1}^{(n)},\,x_k^{(n)})}\), \(k=1,\dots,n\).
On each \(S_n\), integration reduces to a finite sum:
\[
\int_I s(x)\,dx=\sum_{k \in [n]} c_k\,(x_k^{(n)}-x_{k-1}^{(n)}) \quad \text{for } s=\sum_{k \in [n]} c_k\,\chi_{[x_{k-1}^{(n)},\,x_k^{(n)})}.
\]
Refinement maps \(S_n \hookrightarrow S_m\) for \(m\ge n\) turn \((S_n)_{n\in\mathbb N}\) into a directed system whose inductive limit \(\bigcup_{n} S_n\) is the space of step functions, dense in \(\mathcal{F}_{\mathrm{int}}(I)\).
The inner product defined above can be defined on each of these $S_n$ and extends by density to an inner product on \(\mathcal{F}_{\mathrm{int}}(I)\). 

\medskip

On $\mathcal{F}_{\mathrm{int}}(I)$, fix the deformed scalar and addition operations
\[
(\alpha\cdot_{2} f)(x):=\epsilon_{2}(\alpha)\,f(x),
\qquad
(f+_{p/2}g)(x):=\epsilon_{2/p}\big(\epsilon_{p/2}(f(x))+\epsilon_{p/2}(g(x))\big),
\]
for \(\alpha\in F\) and \(f,g\in\mathcal{F}_{\mathrm{int}}(I)\).
On \(\bigl(\mathcal{F}_{\mathrm{int}}(I),+_{p/2},\cdot_{2}\bigr)\), define
\[
\langle f,g\rangle_{\mathcal L^p}
\;:=\;
{}^{^{^{p}}}\!\!\!\!\!\int_{I}\,
\epsilon_{1/2}\!\bigl(f(x)\,\overline{g(x)}\bigr)\,dx.
\]
This is an inner product on $\bigl(\mathcal{F}_{\mathrm{int}}(I),+_{p/2},\cdot_{2}\bigr)$ over $(F, +_{p},\cdot)$ (which is sesquilinear for \(F=\mathbb C\) and bilinear for \(F=\mathbb R\)) in this structure.
The associated norm is
\[
\|f\|_{\mathcal L^p}
\;:=\;
\left(\int_{I} |f(x)|^{p}\,dx\right)^{\!1/p},
\]
and the corresponding space is
\[
\mathcal L^p(F)
\;=\;
\bigl\{\,f\in\mathcal{F}_{\mathrm{int}}(I)\;\big|\;\|f\|_{\mathcal L^p}<\infty\,\bigr\}.
\]

As before, for \(p=2\) the operations \(+_{2}\) and \(\cdot_{2}\) coincide with the classical ones, and
\(\langle f,g\rangle_{\mathcal L^2}=\int_{I} f(x)\,\overline{g(x)}\,dx\).
For \(p\ne 2\), this construction provides a Hilbert framework modeling the \(\mathcal L^p\)-norm via the
\(F\)-line \((F,+_{p},\cdot)\), while the classical norm is not induced by any inner product on the usual
vector space structure, by the Jordan--von Neumann theorem \cite{JordanNeumann}.
\subsection{Recovering infinite norms}
\subsubsection{Computation of certain limits}
In this section, we will study limits involving multiplicative automorphisms, which will enable us to generalize the notion of generalized means. In this first lemma, whose proof is immediate, we observe the discontinuity of the limits  of a continuous multiplicative complex automorphism. This discontinuity will have an influence on the structure studied below.

\noindent Let \(z=x+iy\) with \(x,y\in\mathbb{R}\). Its modulus is \(|z|:=\sqrt{x^{2}+y^{2}}\).
For \(z\neq 0\), an argument of \(z\), denoted $\operatorname{arg} (z)$, is any \(\theta\in\mathbb{R}\) (defined modulo \(2\pi\)) such that
\[
\cos\theta=\frac{x}{|z|}\quad\text{and}\quad \sin\theta=\frac{y}{|z|}.
\]
At \(z=0\), \(\arg z\) is undefined; the principal argument is \(\operatorname{arg} (z)\in(-\pi,\pi]\).


\begin{lemm} \label{limits}
Let $F \in \{\mathbb{R},\mathbb{C}\}$ and $a\in \mathbb{C}$, we write $a = |a|e^{i \theta_a}$ where $\theta_a \in \mathbb{R}$. We have 
\begin{enumerate} 
\item  $ \underset{ |\alpha| \rightarrow 0}{\lim} \epsilon_\alpha ( a)=\left\{ \begin{array}{lll} e^{i \theta_a} & \text{ when } a \neq 0, \\ 0 &  \text{ when }  a=0.  \end{array} \right.$ 
\item  $ \underset{ \substack{ |\alpha| \rightarrow \infty,\\  \operatorname{Re}(\alpha) >0}}{\lim} \epsilon_\alpha ( a)=\left\{ \begin{array}{lll} 0 & \text{ when } |a|<1,\\ e^{i \theta_a} & \text{ when } |a|=1. \end{array} \right.$ 

We note that this limit is not defined when $|a|>1$. 

In particular, when $|a|>1$, $\underset{ \substack{ |\alpha| \rightarrow \infty,\\  \operatorname{Re}(\alpha) <0}}{\lim} |\epsilon_\alpha ( a)|=\infty$. 
\item  $ \underset{ \substack{ |\alpha| \rightarrow \infty,\\  \operatorname{Re}(\alpha) <0}}{\lim} \epsilon_\alpha ( a)=\left\{ \begin{array}{lll} e^{i \theta_a} & \text{ when } |a|=1,\\ 0 &  \text{ when }  |a|>1.  \end{array} \right.$ 

We note that this limit is not defined when $|a|<1$. 

In particular, when $|a|<1$, $\underset{ \substack{ |\alpha| \rightarrow \infty,\\  \operatorname{Re}(\alpha) <0}}{\lim} |\epsilon_\alpha ( a)|=\infty$. 
\end{enumerate}
\end{lemm} 

We have the following result for the limit at infinity when considering the inverse of a continuous multiplicative complex automorphisms:
\begin{lemm}\label{invlimit}
Let $\Theta \in (-\pi /2, \pi /2) \cup (\pi /2, 3\pi/2)$.
\begin{enumerate}
\item When $\Theta \neq 0$, we have
$$\lim_{\substack{|\alpha| \rightarrow \pm \infty \\ \Theta = \operatorname{arg}(\alpha)}}  \epsilon_{\alpha}^{-1} = \epsilon_{- i \tan(\Theta)}.$$
\item When $\Theta = 0$, we have
$$\lim_{\substack{|\alpha| \rightarrow \pm \infty \\ \Theta = \operatorname{arg}(\alpha)}}  \epsilon_{\alpha}^{-1}(a) = \left\{
\begin{array}{lll}
\epsilon_{- i \tan(\Theta)} (a)& \text{when } a \neq 0, \\
0 & \text{otherwise}.
\end{array}
\right.$$
\end{enumerate}
\end{lemm}

\begin{proof} 
The result follows from the fact that
\begin{align*}
      \epsilon_{ \alpha}^{-1}(a) =  \epsilon_{ \frac{1-i \operatorname{Im} (\alpha)}{\operatorname{Re}(\alpha)}} (a) &=   \epsilon_{ \frac{1-i |\alpha| \sin ( \Theta) }{|\alpha| \cos ( \Theta)} } (a)\\
    &=   \epsilon_{ \frac{1/ |\alpha|-i \sin ( \Theta) }{ \cos ( \Theta)} }(a)\\
    &= e^{\frac{1/|\alpha|-i\sin(\Theta)}{\cos(\Theta)}\ln|a| +i\theta_a} \\
    & = |a|^{\frac{1/|a|}{\cos(\Theta)}}|a|^{\frac{-i\sin(\Theta)}{\cos(\Theta)}} e^{i\theta_a} \\
    & = |a|^{\frac{1/|\alpha|}{\cos(\Theta)}}|a|^{-i\tan(\Theta)} e^{i\theta_a}.
\end{align*}
\end{proof} 

For each parameter \(\alpha\) such that there exists a continuous multiplicative automorphism of \(F\) (denoted \(\epsilon_\alpha\)), we obtain an induced addition \(+_\alpha\) on \(F\) that turns \((F,+_\alpha,\cdot)\) into a field. 
Viewing these \((F,+_\alpha,\cdot)\) as a family of field structures indexed by \(\alpha\), we consider the limiting behavior as \(\alpha \to 0\) or \(\alpha \to \pm\infty\). 
What algebraic/analytic system arises in these limits is precisely the question addressed by the next Definition–Lemma. 
In Subsection~\ref{infinitenorm} we will see that these limits yield \(\infty\)-norm inner product near-vector spaces, and in Subsection~\ref{weighted} they also provide a framework for weighted generalized means of families of complex numbers. We recall that a non-associative field is defined analogously to a field, except that the multiplication is not required to be associative.

\begin{definlem}\label{infty}
Let $\Theta \in (-\pi /2, \pi /2) \cup (\pi /2, 3\pi/2)$, and $a_t \in F$ for all $t \in [n]$. We write $a_t =|a_t| e^{i \theta_t}$ where $\theta_t \in \mathbb{R}$. For all $k \in I$, we define $[n]_{k} = \{ t \in [n] \mid |a_t| = |a_{k}| \}$ and $\sum_{t\in [n]_k} e^{i \theta_t} = r_k s_k$ where $r_k \in \mathbb{R}^+$ and $s_k\in \mathbb{S}$.

\begin{enumerate}
\item We define the {\sf limit of the continuous fields at $\infty$ in the direction $\Theta$}, denoted by $F_{\infty,\Theta}$, to be the system $(F, +_{\infty, \Theta},\cdot)$ where for all $a, b \in F$, the operation $+_{\infty, \Theta} $ is defined by
\begin{align*}
a+_{\infty,\Theta} b &= \underset{ \substack{|\alpha| \rightarrow \infty\\ \operatorname{Re}(\alpha) >0 \\ \operatorname{arg} (\alpha)=\Theta}  } {\lim} \epsilon_{\alpha}^{-1}( \epsilon_\alpha( a)  + \epsilon_\alpha (b) ).
\end{align*}
 $(F, +_{\infty, \Theta},\cdot)$ is a non-associative field.

We define
\begin{align*}
{}^{\infty, \Theta} \sum_{k\in [n]} a_k &:= \underset{ \substack{|\alpha| \rightarrow \infty \\ \operatorname{Re}(\alpha) >0 \\ \operatorname{arg} (\alpha)=\Theta} }{\lim} \epsilon_{\alpha}^{-1}\left( \sum_{k\in [n]} \epsilon_\alpha( a_k)  \right)\\
&= |a_M|\,r_M^{-i\operatorname{tan} (\Theta)} s_M.
\end{align*}

We suppose that there exists $M \in [n]$ such that $|a_{M}| = \operatorname{max} \{ |a_k|\,| k \in [n]:  a_k \neq 0 \text{ and } r_k \neq 0  \} $ exists. We have:
\begin{align*}
{}^{\infty, \Theta} \sum_{k\in [n]} a_k &= |a_M|\,r_M^{-i\operatorname{tan} (\Theta)} s_M.
\end{align*}
It is zero otherwise.
\item We define the {\sf limit of the continuous fields at $-\infty$ in the direction $\Theta$}, denoted by $F_{-\infty, \Theta}$, to be the system $(F, +_{-\infty, \Theta},\cdot)$ where for all $a, b \in F$, the operation $+_{-\infty, \Theta}$ is defined by
\begin{align*}
a+_{-\infty, \Theta} b &=\underset{ \substack{|\alpha| \rightarrow \infty\\ \operatorname{Re}(\alpha) <0 \\ \operatorname{arg} (\alpha)=\Theta}  }{\lim} \epsilon_{\alpha}^{-1}( \epsilon_\alpha( a)  + \epsilon_\alpha (b) ).
\end{align*}

$(F, +_{-\infty, \Theta},\cdot)$ is a non-associative field.

We define
\begin{align*}
{}^{-\infty, \Theta} \sum_{k\in [n]} a_k &:=  \underset{ \substack{|\alpha| \rightarrow \infty\\ \operatorname{Re}(\alpha) <0 \\ \operatorname{arg} (\alpha)=\Theta}  }{\lim} \epsilon_{\alpha}^{-1}\left( \sum_{k\in [n]} \epsilon_\alpha( a_k) \right) \\
&= |a_m|r_m^{-i \operatorname{tan} (\Theta) } s_m.
\end{align*}

We suppose that there exists $m \in [n]$, and $|a_{m}| = \operatorname{min} \{ |a_k \mid | k \in [n]:  |a_k|\neq 0 \text{ or } r_k\neq 0 \}$. We have:
\begin{align*}
{}^{-\infty, \Theta} \sum_{k\in [n]} a_k &= |a_m|r_m^{-i \operatorname{tan} (\Theta) } s_m.
\end{align*}
It is zero otherwise.
\end{enumerate}
\end{definlem}

\begin{proof} The fact that the system $(F, +_{\pm \infty, \Theta}, \cdot)$ satisfies all the field axioms except for the one stating the associativity of the operation $+_{\pm \infty, \Theta}$ follows from the properties of multiplicative automorphisms. Multiplicative automorphisms preserve the additive identity, send additive inverses to additive inverses, and also send multiplicative inverses to multiplicative inverses (see also \cite[Lemma 1.2.]{BM2024}). To illustrate the failure of associativity, consider the following counterexample. Suppose $a_{1} = 2$, $a_{2} = 3$, and $a_{3} = -3$. Then, $(2 +_{\infty} 3) +_{\infty} (-3) = 3 +_{\infty} (-3) = 0$
and, $2 +_{\infty} (3 +_{\infty} (-3)) = 2 +_{\infty} 0 = 2$. Clearly, the two results differ, and associativity is broken.

Let $a_1, \ldots, a_n \in F$. We observe that $[n] = [n]_{k_1} \bigcupdot \cdots \bigcupdot  [n]_{k_d}$ for some $k_i \in [n]$, $i \in [d]$, $1 \leq d \leq n$, with $|a_{k_1}| < |a_{k_2}| < \cdots < |a_{k_d}|$.


\begin{enumerate} 
\item Now, suppose that $M$ exists and $M = k_b$ where $b \in [d]$.
\begin{align*}
    \epsilon_{\alpha}^{-1} \left( \sum_{k\in [n]}  \epsilon_{\alpha} (a_k) \right) & = \epsilon_{\alpha}^{-1} \left( \sum_{t \in [b]} |a_{k_t}|^{\alpha} \sum_{l \in [n]_{k_t}} e^{i \theta_{l}} \right) \\
    &= \epsilon_{\alpha}^{-1}\left( \sum_{t\in [b]} |a_{k_t}|^{\alpha} r_{k_l} s_{k_l} \right)\\
    &= \epsilon_{\alpha}^{-1}\left(|a_{M}|^{\alpha} r_{M} s_{M} \left(\sum_{t \in [b-1]} \frac{|a_{k_t}|^{\alpha} r_{k_l} s_{k_l}}{|a_{M}|^{\alpha} r_{M} s_{M}} + 1 \right) \right)\\
    &= |a_M|\, r_M^{\frac{1-i \operatorname{Im} (\alpha)}{\operatorname{Re}(\alpha)}} s_M \epsilon_{\alpha}^{-1} \left( \sum_{t \in [b-1]}  \left(\frac{|a_{k_t}|}{|a_{M}|} \right)^{\alpha} \frac{r_{k_l} s_{k_l}}{ r_M s_M}  +  1 \right).
\end{align*}

Therefore, by Lemma \ref{limits}, since $\frac{|a_{k_t}|}{|a_{M}|} < 1$ for all $t \in [b-1]$, by definition of $M$, we have
$$ \underset{\substack{ |\alpha | \rightarrow \infty ,\\ \operatorname{Re}(\alpha) >0 \\ \operatorname{arg} (\alpha)=\Theta} }{\lim} \left(\frac{|a_{k_t}|}{|a_{M}|} \right)^{\alpha} = 0,$$ 
and thus, by Lemma \ref{invlimit}, since $r_M \neq 0$, by definition of $M$, we have
\begin{align*}
    {}^{\infty , \Theta} \sum_{k\in [n]} a_k &= \underset{\substack{ |\alpha | \rightarrow \infty ,\\ \operatorname{Re}(\alpha) >0 \\ \operatorname{arg} (\alpha)=\Theta}  }{\lim} |a_M|\, r_M^{\frac{1-i \operatorname{Im} (\alpha)}{\operatorname{Re}(\alpha)}} s_M \epsilon_{\alpha}^{-1} \left( \sum_{t \in [b-1]}  \left(\frac{|a_{k_t}|}{|a_{M}|} \right)^{\alpha} \frac{r_k s_k}{ r_M s_M}  +  1 \right)\\
    &= |a_M|\, r_M^{-i \operatorname{tan} (\Theta)}  s_M.
\end{align*}

\item Suppose that $m$ exists and $m= k_c$ where $c \in [n]$. By definition of $m$, we have for all $j \leq c$,
\begin{align*}
    \sum_{k \in [n]_{k_j}} \epsilon_\alpha ( a_k) &= |a_{k_j}|^\alpha \sum_{s \in [n]_{k_j}}  e^{i \theta_{s}} = 0, \text{and } a_{k} \neq 0 \text{ for all } k \in [n]_{k_j}, \text{ for all } j \in [k_c].
\end{align*}
We denote $[n]_{m} = [n]_{k_c} \bigcupdot \cdots \bigcupdot [n]_{k_d}$, so that
\begin{align*}
    \epsilon_{\alpha}^{-1} \left( \sum_{k \in [n]}  \epsilon_{\alpha} (a_k) \right)
    &= \epsilon_{\alpha}^{-1} \left( \sum_{l=c}^d |a_{k_l}|^{\alpha} \sum_{t \in [n]_{k_l}} e^{i \theta_{t}} \right)\\
    &= \epsilon_{\alpha}^{-1}\left( \sum_{l=c+1}^d |a_{k_l}|^{\alpha} r_{k_l} s_{k_l} \right)\\
    &= \epsilon_{\alpha}^{-1}\left(|a_{m}|^{\alpha} r_{m} s_{m} \left(\sum_{l=c+1}^d\frac{|a_{k_l}|^{\alpha} r_{k_l} s_{k_l}}{|a_{m}|^{\alpha} r_{m} s_{m}} + 1 \right) \right)\\
    &= |a_m|\, r_m^{\frac{1-i \operatorname{Im} (\alpha)}{\operatorname{Re}(\alpha)}} s_m \epsilon_{\alpha}^{-1} \left( \sum_{l=c+1}^d \left(\frac{|a_{k_l}|}{|a_{m}|} \right)^{\alpha}    \frac{r_{k_l} s_{k_l}}{ r_m s_m}  +  1 \right).
\end{align*}

Therefore,  by Lemma \ref{limits}, since $\frac{|a_{k_l}|}{|a_{m}|} > 1$, for all $l \in \{ c+1, \cdots , d\}$, by definition of $m$, we have
$$ \underset{  \substack{|\alpha | \rightarrow -\infty ,\\ \operatorname{Re}(\alpha) <0 \\ \operatorname{arg} (\alpha)=\Theta} }{\lim} \left(\frac{|a_{k_l}|}{|a_{m}|} \right)^{\alpha} = 0,$$ 
and thus, by Lemma \ref{invlimit}, since $r_m \neq 0$, by definition of $m$, we have
\begin{align*}
    {}^{-\infty , \Theta} \sum_{ k\in [n]} a_k &= \underset{ \substack{|\alpha | \rightarrow -\infty ,\\ \operatorname{Re}(\alpha) <0 \\ \operatorname{arg} (\alpha)=\Theta} }{\lim} |a_m|\, r_m^{\frac{1-i \operatorname{Im} (\alpha)}{\operatorname{Re}(\alpha)}} s_m \epsilon_{\alpha}^{-1} \left( \sum_{l=c+1}^d \left(\frac{|a_{k_l}|}{|a_{m}|} \right)^{\alpha}    \frac{r_{k_l} s_{k_l}}{ r_m s_m}  +  1 \right)\\
    &= |a_m|\, r_m^{-i \operatorname{tan} (\Theta)}  s_m.
\end{align*}
\end{enumerate}
\end{proof}

\begin{rem}
We maintain the notations from Definition-Lemma \ref{infty}. 
\begin{enumerate}
\item Let $a = |a| e^{i \theta_a}$ and $b = |b| e^{i \theta_b}$ be complex numbers where $\theta_a, \theta_b \in \mathbb{R}$. We have:
\begin{enumerate}
\item $a +_{\infty} b = 0$ (resp. $a +_{-\infty} b = 0$) if and only if $a = -b$;
\item when $a \neq -b$,
\begin{itemize}
\item if $|a| \neq |b|$, then $a +_{\infty} b = c$ (resp. $a +_{-\infty} b = c$) where $c \in \{a, b\}$ and $|c| = \max\{|a|, |b|\}$ (resp. $|c| = \min\{|a|, |b|\}$);
\item if $|a| = |b|$, we have:
$$a +_{\infty} b = |a| \sqrt{2(1 + \cos(\theta_a - \theta_b))}^{-i\operatorname{tan} (\Theta)} \operatorname{sgn}\left(2\cos\left(\frac{\theta_a - \theta_b}{2} \right)\right) e^{i\frac{\theta_a + \theta_b}{2}}.$$
Indeed, we have \begin{align*}
e^{i \theta_a} + e^{i \theta_b} &= e^{i\frac{\theta_a + \theta_b}{2}} \left(e^{i\frac{\theta_a - \theta_b}{2}} + e^{-i\frac{\theta_a - \theta_b}{2}}\right) = e^{i\frac{\theta_a + \theta_b}{2}} 2 \cos\left(\frac{\theta_a - \theta_b}{2}\right) \\
\end{align*}
and so 
$$r_m = r_M = \left|2 \cos\left(\frac{\theta_a - \theta_b}{2}\right)\right| = \sqrt{2(1 + \cos(\theta_a - \theta_b))}$$
and
$$s_m = s_M = \operatorname{sgn}\left(2\cos\left(\frac{\theta_a - \theta_b}{2} \right)\right) e^{i\frac{\theta_a + \theta_b}{2}};$$

\end{itemize}
\end{enumerate}
\item at $\Theta= -\pi / 2$, $\Theta = \pi/2$, and $\Theta = 3\pi/2$, the limit at $\infty$ and $-\infty$ is indeterminate;
\item when $F = \mathbb{R}$, $r_{M} = \left|\sum_{i \in [n]_M} \operatorname{sgn}(a_i) \right|$, $r_{m} = \left|\sum_{i \in [n]_m} \operatorname{sgn}(a_i) \right|$ and  $s_M = \operatorname{sgn} \left(\sum_{i \in [n]_M} \operatorname{sgn}(a_i)\right)$, $s_m = \operatorname{sgn} \left(\sum_{i \in [n]_m} \operatorname{sgn}(a_i)\right)$. Moreover, when all $a_i$ are non-negative real numbers and $\Theta$ is zero, we obtain that the limit at $\infty$ (resp. $-\infty$) is the minimum (resp. maximum) of the values times the number of times the maximum (resp. the minimum) appears.
\end{enumerate}
\end{rem}
We shall now proceed to define the limit systems for continuous fields at the origin.
\begin{definlem} \label{zerolimit} 
Let $\Theta \in (-\pi /2, \pi /2) \cup (\pi /2, 3\pi/2)$, and for all $k \in I$, let $a_k$ be elements in $F$, and define $[n]_0= \{ k \in [n] \mid a_k \neq 0\}$. We can express $a_k$ as $|a_k| e^{i \theta_k}$, where $\theta_k \in \mathbb{R}$ for all $k \in [n]$ and $$\sum_{k\in [n]_0} e^{i \theta_k} = r e^{i \theta},$$ where $r \in \mathbb{R}^{+} \setminus \{0\}$ and $\theta \in \mathbb{R}$. We introduce {\sf the limit of continuous fields at $0$ with direction $\Theta$}, denoted as $F_{0}$, as the system $(F, +_{0, \Theta}, \cdot)$, where for all $a, b \in F$, the operation $+_{0, \Theta}$ is defined as follows:
$$
a +_{0, \Theta} b = \lim_{\substack{|\alpha| \rightarrow 0 \\ \operatorname{arg} (\alpha)=\Theta }} \epsilon_{\alpha}^{-1}( \epsilon_\alpha( a)  + \epsilon_\alpha (b) ) =\lim_{\substack{|\alpha| \rightarrow 0 \\ \operatorname{arg} (\alpha)=\Theta }} \overline{\epsilon_{\alpha}}^{-1}( \overline{\epsilon_\alpha}( a)  + \overline{\epsilon_\alpha} (b) ).
$$ 

We also define
$${}^{0,\Theta} \sum_{k\in [n]} a_k := \lim_{\substack{|\alpha| \rightarrow 0 \\ \operatorname{arg} (\alpha)=\Theta }} \epsilon_{\alpha}^{-1}\left( \sum_{k\in [n]} \epsilon_\alpha( a_k) \right).$$

When $r=1$, we obtain $\sum_{k\in [n]}\cos(\theta - \theta_k) = 1$, and 
$${}^{0, \Theta} \sum_{k\in [n]} a_k = \prod_{ k \in [n]_0} |a_k|^{ \cos(\theta - \theta_k)(1+i\tan(\Theta))} e^{i\theta }.$$ 

In any other scenarios, the absolute value of ${}^{0, \Theta} \sum_{k \in [n]} a_k$ is infinity.

\end{definlem}
\begin{proof} 

Let $a_1, \cdots, a_n \in F$. To prepare for our proof, we first establish some essential results. When $\operatorname{Re}(\alpha) > 0$ and $|\alpha| \rightarrow 0$, we observe that $\frac{1- i \operatorname{Im}(\alpha)}{\operatorname{Re}(\alpha)}$ tends towards infinity. Similarly, for $\operatorname{Re}(\alpha) < 0$, $\frac{1- i \operatorname{Im}(\alpha)}{\operatorname{Re}(\alpha)}$ approaches negative infinity as $|\alpha| \rightarrow 0$.

In the following, we write $\sum_{k \in [n]} \epsilon_{\alpha}(a_k) = r_{\alpha}  e^{i \theta_{\alpha}}$ where $r_{\alpha}>0$ and $\theta_{\alpha}\in \mathbb{R}$. 
We have:
$$\epsilon_{\alpha}^{-1}\left(\sum_{k \in [n]} \epsilon_{\alpha}(a_k)\right) = r_{\alpha}^{\frac{1- i \operatorname{Im}(\alpha)}{\operatorname{Re}(\alpha)}} e^{i \theta_{\alpha}} = r_{\alpha}^{\frac{1}{\operatorname{Re}(\alpha)}} r_{\alpha}^{-i \tan \Theta} e^{i \theta_{\alpha}},$$
where $\theta_{\alpha} \in \mathbb{R}$. 
We define $[n]_0 = \{ k \mid a_k \neq 0 \}$, and through straightforward computation, we find that
\begin{align*}
    r_{\alpha} & = \sqrt{\left(\sum_{k \in [n]_0} |a_k|^{\alpha} \cos(\theta_k)\right)^2 + \left(\sum_{k \in [n]_0} |a_k|^{\alpha} \sin(\theta_k)\right)^2} = \sqrt{x_{\alpha}^{2} + y_{\alpha}^{2} },
\end{align*}
where $x_{\alpha} = \sum_{k \in [n]_0} |a_k|^{\alpha}\cos(\theta_k)$ and $y_{\alpha} = \sum_{k \in [n]_0} |a_k|^{\alpha}\sin(\theta_k)$.

Then,  
\begin{align*}
    &\lim_{\substack{|\alpha| \rightarrow 0 \\ \operatorname{arg} (\alpha)=\Theta }} r_{\alpha}\\ 
    & = \lim_{\substack{|\alpha| \rightarrow 0 \\ \operatorname{arg} (\alpha)=\Theta }} \left(\left(\sum_{k \in [n]_0} |a_k|^{\alpha} \cos(\theta_k)\right)^2 + \left(\sum_{k \in [n]_0} |a_k|^{\alpha} \sin(\theta_k)\right)^2 \right)^{1/2} \\
    & =  \lim_{\substack{|\alpha| \rightarrow 0 \\ \operatorname{arg} (\alpha)=\Theta }}   \left(\left(\sum_{k \in [n]_0}\lim_{\substack{|\alpha| \rightarrow 0 \\ \operatorname{arg} (\alpha)=\Theta }} |a_k|^{|\alpha|e^{i \theta_{\alpha}}} \cos(\theta_k)\right)^2 + \left(\sum_{k \in [n]_0}\lim_{\substack{|\alpha| \rightarrow 0 \\ \operatorname{arg} (\alpha)=\Theta }} |a_k|^{|\alpha|e^{i \theta_\alpha}} \sin(\theta_k)\right)^2 \right)^{1/2} \\
    & = \left(\left(\sum_{k \in [n]_0}\cos(\theta_k)\right)^2 + \left(\sum_{k \in [n]_0}\sin(\theta_k)\right)^2 \right)^{1/2} \\
    & = r.
    \end{align*}

Next, we want to compute the limit as $|\alpha| \rightarrow 0$ for the term $e^{i\theta_{\alpha}}$. We note that 
$$e^{i \theta_{\alpha}}=\frac{\sum_{k \in [n]} \epsilon_{\alpha}(a_k)}{\left|\sum_{k\in [n]} \epsilon_{\alpha}(a_k)\right|}.$$
Hence, 
\begin{align*}
    \lim_{\substack{|\alpha| \rightarrow 0 \\ \operatorname{arg} (\alpha)=\Theta }}e^{i\theta_{\alpha}} & = \lim_{\substack{|\alpha| \rightarrow 0 \\ \operatorname{arg} (\alpha)=\Theta }} \frac{1}{r_\alpha} \sum_{k \in [n]}\epsilon_{\alpha}(a_k) \\
    & =\frac{1}{r}\lim_{\substack{|\alpha| \rightarrow 0 \\ \operatorname{arg} (\alpha)=\Theta }}  \sum_{k \in [n]}\epsilon_{\alpha}(a_k) \\
    & = \frac{1}{r} \lim_{\substack{|\alpha| \rightarrow 0 \\ \operatorname{arg} (\alpha)=\Theta }} \sum_{k \in [n]}|a_k|^{|\alpha|e^{i\theta_\alpha}}e^{i\theta_k} \\
    & = \frac{1}{r} \sum_{k \in [n]_0}e^{i\theta_k} =
\frac{e^{i \theta}}{r}.
\end{align*}

%

Now, we analyze various cases for different values of the limit $\epsilon_{\alpha}^{-1}\left(\sum_{k \in [n]} \epsilon_{\alpha}(a_k)\right) $ as $ |\alpha| \rightarrow 0$:
\begin{enumerate}
\item When $r < 1$, $\lim_{\substack{|\alpha| \rightarrow 0 \\ \operatorname{arg} (\alpha)=\Theta }} |\epsilon_{\alpha}^{-1}\left(\sum_{k\in [n]} \epsilon_{\alpha}(a_k)\right)| = 0$. 
\item When $r > 1$,  $\lim_{\substack{|\alpha| \rightarrow 0 \\ \operatorname{arg} (\alpha)=\Theta }}|\epsilon_{\alpha}^{-1}\left(\sum_{k\in [n]} \epsilon_{\alpha}(a_k)\right)| =\infty$. 
\item When $r = 1$, we have $\lim_{\substack{|\alpha| \rightarrow 0 \\ \operatorname{arg} (\alpha)=\Theta }}r_{\alpha}^{-i \tan \Theta}=1$. We can express $r_{\alpha}^{\frac{1}{\operatorname{Re}(\alpha)}}$ as $e^{\frac{\ln(r_{\alpha})}{\operatorname{Re}(\alpha)}}$. We have $\operatorname{lim}_{|\alpha | \rightarrow 0 } \operatorname{ln} ( r_\alpha)  =0$ and $\operatorname{lim}_{|\alpha | \rightarrow 0 } \operatorname{Re}(\alpha)=0$, so we can apply L'Hopital's rule on $\frac{  \operatorname{ln} ( r_\alpha ) }{\operatorname{ Re}(\alpha)} = \frac{ \operatorname{ln} ( r_\alpha ) }{|\alpha| \cos (\Theta)}$ on the variable $|\alpha|$. We first calculate the derivative of $\ln(r_\alpha)$:
\begin{align*}
   & \frac{d}{d|\alpha|}\ln(r_\alpha) \\
   & =  \frac{1}{2r_{\alpha}^2} \frac{d}{d|\alpha|} \left(x_{\alpha}^2 + y_{\alpha}^2 \right)  \\
   & = \frac{1}{2r_{\alpha}^2}  \left( 2 x_{\alpha} \left( \sum_{j \in [n]_0} \ln(|a_j|)|a_j|^{|\alpha|e^{i \Theta}}\cos(\theta_j) \right) + 2 y_{\alpha} \left(\sum_{j \in [n]_0} \ln(|a_j|)|a_j|^{|\alpha|e^{i \Theta}}\sin(\theta_j) \right) \right)   \\
   & = \frac{1}{r_{\alpha}^2}  \sum_{j \in [n]_0} \ln(|a_j|)|a_j|^{|\alpha|e^{i \Theta}} \left(\cos(\theta_j) x_{\alpha} +\sin(\theta_j)y_{\alpha} \right).
\end{align*}

We obtain:
\begin{align*}  & \lim_{\substack{|\alpha| \rightarrow 0 \\ \operatorname{arg} (\alpha)=\Theta } }  \frac{  \operatorname{ln} ( r_\alpha ) }{|\alpha| \cos (\Theta)} \\
& 
 = \lim_{\substack{|\alpha| \rightarrow 0 \\ \operatorname{arg} (\alpha)=\Theta }} \frac{  \sum_{j \in [n]_0} e^{i \Theta}\operatorname{ln} (|a_j|) |a_j|^{|\alpha| e^{i \Theta}} (\operatorname{cos}( \theta_j) x_{\alpha}  +  \operatorname{\sin} ( \theta_j) y_{\alpha}  )}{ r_\alpha^2 \cos (\Theta)} \\
 &=  \frac{e^{i \Theta}}{\cos (\Theta)} \frac{\sum_{j \in [n]_0}  \operatorname{ln} (|a_j|) (\operatorname{cos}( \theta_j) \sum_{k \in [n]_0} \cos( \theta_k)  + \sin ( \theta_j) \sum_{k \in [n]_0}  \sin( \theta_k) )}{\left(\sum_{k \in [n]_0} \cos (\theta_k) \right)^2 + \left(\sum_{k \in [n]_0} \sin (\theta_k) \right)^2}  \\
 &=  \frac{e^{i \Theta}}{\cos (\Theta)}  \sum_{k \in [n]_0}  \ln( |a_k|) m_k 
 \end{align*}
where $m_k:=\operatorname{ cos}(\theta_k) (\sum_{j \in [n]_0} \cos(\theta_j) ) + \sin(\theta_k)  (\sum_{j \in [n]_0} \sin(\theta_j) )$. Notice that 
\begin{align*}
    m_{k} &= \sum_{j \in [n]_0}(\cos(\theta_k)\cos(\theta_j) + \sin(\theta_k)\sin(\theta_j)) \\
    & = \sum_{j \in [n]_0} \cos(\theta_j-\theta_k) \\
    & = \operatorname{Re} \left(\sum_{j \in [n]_0}e^{i(\theta_j-\theta_k)} \right).
\end{align*}Since $r=1$, we have $\sum_{j \in [n]_0}e^{i\theta_j} = e^{i\theta}$ and so, we have $$m_k = \operatorname{Re}\left( \sum_{j \in [n]_0} e^{i (\theta_j- \theta_k )}\right) = \operatorname{Re}\left( e^{i (\theta - \theta_k) } \right)= \cos( \theta - \theta_k).$$ Thus, 
$$  \lim_{\substack{|\alpha| \rightarrow 0 \\ \operatorname{arg} (\alpha)=\Theta } }r_\alpha^{\frac{1}{\operatorname{Re}(\alpha)} }=  \prod_{ k \in [n]_0} |a_k|^{\cos( \theta - \theta_k)(1+i \tan(\Theta))} .$$ Finally, $\sum_{k \in [n]_0} m_k = \operatorname{Re}\left(\sum_{k \in [n]_0} \left(  \sum_{j \in [n]_0} e^{i (\theta_k -\theta_j)}\right) \right)=1$.
\end{enumerate}

\end{proof} 

\begin{rem}
We maintain the notation from Lemma \ref{zerolimit}, and further analyze the properties of the expressions involved:
\begin{enumerate}
\item We introduce the notation $m_k := \cos(\theta_k) \left(\sum_{j\in [n]} \cos(\theta_j)\right) + \sin(\theta_k) \left(\sum_{i \in [n]} \sin(\theta_j)\right)$. We can express $r^2$ as follows:
\begin{align*}
r^2 = \left| \sum_{k \in [n]_0} e^{i \theta_k} \right|^2 = \left(\sum_{k \in [n]_0}e^{i\theta_k} \right) \overline{\left(\sum_{k \in [n]_0}e^{i\theta_k} \right)}  = \sum_{k \in [n]}\sum_{j \in [n]}e^{i(\theta_k-\theta_j)}. 
\end{align*}
When $k=j$, we get $e^{i(\theta_k-\theta_j)} = 1$. Thus, we have 
\begin{align*}
    r^2 = n + \sum_{k \in [n]}\sum_{j \neq k}e^{i(\theta_k-\theta_j)} = n+ \sum_{k \in [n]}\sum_{j \neq k} \cos(\theta_k-\theta_j).  
\end{align*}
Since $e^{i(\theta_k-\theta_j)}+e^{-i(\theta_k-\theta_j)} = 2\cos(\theta_k - \theta_j)$, we obtain
\begin{align*}
    n+ \sum_{k \in [n]}\sum_{j \neq k} \cos(\theta_k-\theta_j) = n+2 \sum_{k \in [n]}\left(\sum_{j>k}\cos(\theta_j-\theta_k)\right).
\end{align*}
Thus, we have 
\begin{align*}
    r = \sqrt{n+2 \sum_{k \in [n]}\left(\sum_{j>k}\cos(\theta_j-\theta_k)\right)}.
\end{align*}
Similarly, we can express $m_k$ as:
$$
m_k = \operatorname{Re}\left(\sum_{j \in [n]} e^{i (\theta_k - \theta_j)}\right) = \sum_{j\in [n]} \cos(\theta_j - \theta_k) = 1 + \sum_{j \neq k} \cos(\theta_j - \theta_k) = 1 + 2\sum_{j > k} \cos(\theta_j - \theta_k).
$$
This leads to the following observations: 
\begin{itemize}
\item $r=1$ if and only if there exists $\theta \in \mathbb{R}$ such that $\sum_{k \in [n]} e^{i \theta_k} = e^{i \theta}$. This is also equivalent to $\sum_{k \in [n]} \left(\sum_{j > k} \cos(\theta_j - \theta_k)\right) = \frac{1-n}{2}$;
\item $r=0$ if and only if $\sum_{k \in [n]} e^{i \theta_k} = 0$. This is also equivalent to $\sum_{k \in [n]} \left(\sum_{j > k} \cos(\theta_j - \theta_k)\right) = -\frac{n}{2}$.
\end{itemize} 
\item When $e^{i \theta_k} \in \{ \pm 1\}$, we can make the following observations:
\begin{itemize}
\item  For odd values of $n$, $r=0$ is impossible, and $r=1$ if and only if $\frac{n-1}{2}$ of the $\theta_k$'s are $-1$, and $\frac{n+1}{2}$ of the $\theta_k$'s are $1$;
\item for even values of $n$, $r=0$ if and only if half of the $\theta_k$'s are $-1$, and the rest are $1$, and $r=1$ is impossible.
\end{itemize} 
\item  In the case where $F = \mathbb{R}$ and $\theta_k = \frac{2\pi k}{n}$, we find that:
\begin{itemize}
\item $\sum_{k \in [n]} e^{i\theta_k} = 0$, which implies $r=0$;
\item $\sum_{k \in [n]} e^{i(\pi-\theta_k)} = 1$, 
\end{itemize}  
providing a solution for an $n-1$ sum with $r=1$. This is related to an open problem in number theory concerning vanishing sums of roots of unity (cf. \cite{ Lenstra}).
\end{enumerate}
\end{rem}
\subsubsection{Recovering infinite norms from limits}
\label{infinitenorm}

In this section, we explore how one can recover the norms at infinity from inner products defined on certain non-associative near-vector spaces. These spaces satisfy all the axioms of a near-vector space except for associativity, and their values lie in a non-associative \(F\)-like line — a limit of \(F\)-like lines that ultimately forms a non-associative field.

We consider the limits at \(+\infty\) and \(-\infty\) of the finite inner products and norms introduced in Section~\ref{Lpnorms}. Throughout, we fix \(F \in \{\mathbb{R},\mathbb{C}\}\) and \(I \in \{[n],\mathbb{N}\}\). 

We take the product \(F^I\) with componentwise addition \(+_{\infty,0}\) and componentwise scalar multiplication \(\cdot_2\). This construction defines a non-associative, strongly regular near-vector space over \((F,\cdot)\).

Taking \(p \to \infty\) in Definition–Lemma~\ref{infty} (1), we obtain for all \(u = (u_i)_{i \in I},v = (v_i)_{i \in I} \in F^I\),
\[
\langle u,v\rangle_{\ell^\infty}
\;=\;
\lim_{p \to \infty}
\epsilon_p^{-1}\!\left(
  \sum_{i\in I}\epsilon_{p/2}\!\bigl(u_i\,\overline{v_i}\bigr)
\right)
\;=\;
|u_M\overline{v_M}|^{1/2}\,s_M,
\]
where \(M\in I\) is chosen so that \(|u_M\overline{v_M}|=\sup_{i\in I}\{|u_i\overline{v_i}|\}\), and
\[
\sum_{t\in [n]_M} e^{i\theta_t} \;=\; r_M\,s_M,
\qquad
r_M\in\mathbb{R}^+,\ \ s_M\in\mathbb{S},
\]
with the index set of maximizers
$
[n]_M \;:=\; \{\,i\in I \mid |u_i\overline{v_i}|=|u_M\overline{v_M}|\,\}
$
and the corresponding norm
\[
\|u\|_{\ell^{\infty}}
\;=\;
\lim_{p \to \infty}
\left(\sum_{i\in I}|u_i|^p\right)^{1/p}
\;=\;
\sup_{i\in I}\{|u_i|\}.
\]
While in the classical vector–space setting these constructions do not yield a genuine inner product, in the near–vector–space setting they do: all inner–product axioms pass to the limit, and the only potential issue—positive definiteness—can be checked directly from the norm at infinity. Thus, even “at infinity,” the near–vector space framework retains a coherent inner–product–type structure valued in a non-associative field obtained as a limit of \(F\)-like lines. In particular, we recover \[\ell^\infty(F)= \{ (u_i)_{i\in I} | \sup_{i\in I}\{|u_i|\} < \infty \} \] with a pairing whose induced norm is the usual \(\ell^\infty\) norm, now taking values in the non-associative field \((F,+_{\infty},\cdot)\): the form \(\langle -,- \rangle_{\ell^{\infty}}\) defines an inner–product–type structure on \((F^I,+_{\infty,0},\cdot)\) over \((F,+_{\infty},\cdot)\), and \(\|-\|_{\ell^{\infty}}\) is its associated norm.

We take the product \(F^{I}\) with componentwise addition \(+_{-\infty,0}\) and componentwise scalar multiplication \(\cdot_{2}\). This yields a non-associative, strongly regular near–vector space over \((F,\cdot)\).

Taking \(p\to -\infty\), by Definition–Lemma~\ref{infty} (2) we obtain, for \(u = (u_i)_{i \in I},v = (v_i)_{i \in I} \in F^{I}\),
\[
\langle u,v\rangle_{\ell^{-\infty}}
=
\lim_{p\to -\infty}
\epsilon_{p}^{-1}\!\left(\sum_{i\in I}\epsilon_{p/2}\bigl(u_{i}\,\overline{v_{i}}\bigr)\right)
= |u_m\overline{v_m}|^{1/2}s_m,
\]
where $m \in I$ is chosen so that $|u_m\overline{v_m}| = \inf_{i\in I}\,\{ \bigl|u_{i}\,\overline{v_{i}}\bigr|\}$, and \[
\sum_{t\in [n]_m} e^{i\theta_t} \;=\; r_m\,s_m,
\qquad
r_m\in\mathbb{R}^+,\ \ s_m\in\mathbb{S},
\]
with the index set of minimizers
$[n]_m \;:=\; \{\,i\in I \mid |u_i\overline{v_i}|=|u_m\overline{v_m}|\,\}$
and the corresponding norm
\[
\|u\|_{\ell^{-\infty}}
=
\lim_{p\to -\infty}\!\left(\sum_{i\in I}|u_{i}|^{p}\right)^{1/p}
=
\inf_{i\in I}\{|u_{i}|\}.
\]

All the inner–product axioms remain true except positive definiteness, which fails at \(-\infty\).
Consequently, \(\langle-,-\rangle_{\ell^{-\infty}}\) is an indefinite inner–product–type pairing on
\((F^{I},+_{-\infty,0},\cdot)\) (valued in \((F,+_{-\infty},\cdot)\)), and \(\|-\|_{\ell^{-\infty}}\) is only a
semi-norm on \((F^{I},+_{-\infty,0},\cdot)\). 
Accordingly, the space at negative infinity can be defined as
\[
\ell^{-\infty}(F)
\;=\;
\bigl\{\, (u_i)_{i\in I}\ \big|\ \inf_{i\in I}\{|u_i|\} > -\infty \,\bigr\}.
\]
We can also consider the limit of the mixed pairing \(\langle A,B\rangle_{\ell^{p,q}}\) as \(p,q \to \infty\), or the iterated limits obtained by fixing one variable and letting the other tend to infinity — that is, with \(p\) fixed and \(q \to \infty\), or with \(q\) fixed and \(p \to \infty\).
Concretely, leading to the norms: 
\[
\|A\|_{\ell^{\infty,\infty}}
=\lim_{q\to\infty}\lim_{p\to\infty}
\left(\sum_{j \in [n]}\Big(\sum_{i \in [m]}|a_{i,j}|^{p}\Big)^{q/p}\right)^{\!1/q}
=\max_{(i,j) \in [m] \times [n]}\{|a_{i,j}|\}.
\]
Accordingly, the space at infinity can be defined as 
\[\ell^{\infty,\infty}(F^{m \times n}) = \left\{ (a_{i,j}) \Big\mid \max_{(i,j) \in [m] \times [n]}\{|a_{i,j}|\}< \infty\}\right\}. \]
More generally, with our column–first convention,
\[
\|A\|_{\ell^{p,\infty}}
=\lim_{q\to\infty}\|A\|_{\ell^{p,q}}
=\max_{j \in [n]} \left\{\left(\sum_{i \in [m]}|a_{i,j}|^{p}\right)^{\!1/p}\right\}.
\]
Dually,
\[
\|A\|_{\ell^{\infty,q}}
=\lim_{p\to\infty}\|A\|_{\ell^{p,q}}
=\left(\sum_{j \in [n]}\big(\max_{i \in [m]}\{|a_{i,j}|\}\big)^{q}\right)^{\!1/q}.
\]
For the composite inner product corresponding to these norms, we leave it to the reader to deduce the expression from the constructions above.

On \(\mathcal{F}_{\mathrm{int}}(I)\), fix the deformed scalar multiplication and addition
\[
(\alpha \cdot_{2} f)(x) := \epsilon_{2}(\alpha)\,f(x),
\qquad
(f +_{\infty,0} g)(x) := f(x) +_{\infty,0} g(x),
\]
for \(\alpha\in F\) and \(f,g\in\mathcal{F}_{\mathrm{int}}(I)\).
With these operations, \(\mathcal{F}_{\mathrm{int}}(I)\) acquires a (possibly non-associative) near–vector space structure.

For \(f,g\in\mathcal{F}_{\mathrm{int}}(I)\), set \(\varphi(x):=f(x)\,\overline{g(x)}\).
Assume there exists \(x_M\in I\) such that
\[
|\varphi(x_M)| \;=\; \operatorname*{ess\,sup}_{x\in I}\{ |\varphi(x)|\}.
\]
Let the maximizer set be
\[
[n]_M \;:=\; \{\,x\in I : |\varphi(x)| = |\varphi(x_M)|\,\}.
\]
For each \(x\in [n]_M\), write \(\varphi(x)=|\varphi(x_M)|\,e^{i\theta_x}\).
Since the maximizer set \([n]_M\) may be infinite, we impose the following standing assumption: the series
\(\sum_{t\in [n]_M} e^{i\theta_t}\) converges in \(\mathbb{C}\). Under this assumption, define
\[
\sum_{t\in [n]_M} e^{i\theta_t} \;=\; r_M\,s_M,
\qquad
r_M\in\mathbb{R}_{\ge 0},\ \ s_M\in \mathbb{S}:=\{z\in\mathbb{C}:|z|=1\},
\]
i.e., \(s_M\) is the phase of the aggregate of the maximizing directions.

Then, the \(\mathcal{L}^{\infty}\)-pairing is given by
\[
\big\langle f, g \big\rangle_{\mathcal{L}^{\infty}}
\;:=\;
\lim_{p\to\infty}
\epsilon_{1/p}\!\left(\int_{I} \epsilon_{p/2}\!\bigl(\varphi(x)\bigr)\,dx\right)
\;=\;
\bigl|\varphi(x_M)\bigr|^{1/2}\, s_M
\;=\;
\bigl|f(x_M)\,\overline{g(x_M)}\bigr|^{1/2}\, s_M.
\]

Next, we recover the \(\mathcal{L}^{\infty}\) norm. By Definition–Lemma~\ref{infty}, we have
\[
\|f\|_{\mathcal L^{\infty}}
\;:=\;
\lim_{p \to \infty}
\left(\int_{I} |f(x)|^{p}\,dx\right)^{\!1/p}
\;=\;
\operatorname*{ess\,sup}_{x\in I}\{|f(x)|\}.
\]
(See, e.g., \cite{Fitz} for the classical \(\mathcal{L}^{p}\to\mathcal{L}^{\infty}\) limit.)

Accordingly, the corresponding space \(\mathcal L^{\infty}(F)\) can be described as
\[
\mathcal L^{\infty}(F)
\;=\;
\Bigl\{\, f\in \mathcal{F}_{\mathrm{int}}(I)\ \Big|\ 
\operatorname*{ess\,sup}_{x\in I}\{\big|\bigl(f(x)\overline{g(x)}\bigr)\big| <\infty \Bigr\}.
\]
The pairing \(\langle -,-\rangle_{\mathcal L^{\infty}}\) endows the above structure with an
inner–product–type non-associative near–space over the non-associative near-line
\((F,+_{\infty,0},\cdot)\), and its induced norm is \(\|-\|_{\mathcal L^{\infty}}\).

On \(\mathcal{F}_{\mathrm{int}}(I)\), fix the deformed scalar and addition operations
\[
(\alpha\cdot_{2} f)(x):=\epsilon_{2}(\alpha)\,f(x),
\qquad
(f+_{-\infty,0}g)(x):=f(x)+_{-\infty,0}g(x),
\]
for \(\alpha\in F\) and \(f,g\in\mathcal{F}_{\mathrm{int}}(I)\).
With these operations, \(\mathcal{F}_{\mathrm{int}}(I)\) acquires the structure of a non-associative near–vector space.

For \(f,g\in\mathcal{F}_{\mathrm{int}}(I)\), we follow the same notations as above. Assume there exists $x_m \in I$ such that 
\[ |\varphi(x_m)| = \operatorname*{ess\,inf}_{x\in I}\{|f(x)|\}.\]
Let the minimizer set be 
\[
[n]_m \;:=\; \{\,x\in I : |\varphi(x)| = |\varphi(x_m)|\,\}.
\]
For each $x \in [n]_m$, we write $\varphi(x) = |\varphi(x_m)|e^{i\theta_x}$. Since the minimizer set $[n]_m$ may be infinite, we impose the following standing assumption: the series
\(\sum_{t\in [n]_M} e^{i\theta_t}\) converges in \(\mathbb{C}\). Under this assumption, define
\[
\sum_{t\in [n]_m} e^{i\theta_t} \;=\; r_m\,s_m,
\qquad
r_m\in\mathbb{R}_{\ge 0},\ \ s_m\in \mathbb{S}:=\{z\in\mathbb{C}:|z|=1\},
\]
i.e., \(s_m\) is the phase of the aggregate of the minimizing directions.
Then, the $\mathcal{L}^{\infty}$-pairing is given by

\[\langle f, g \rangle_{\mathcal L^{-\infty}}
\;:=\;
\lim_{p \to -\infty}\,
\epsilon_{1/p}\!\left(\int_{I} \epsilon_{p/2}\!\bigl(f(x)\,\overline{g(x)}\bigr)\,dx\right)
\;= |\varphi(x_m)|^{1/2}s_m = \; \big|\bigl(f(x_m)\overline{g(x_m)}\bigr)\big|^{1/2} s_{m}.
\]

All the inner-product axioms remain true except for positive definiteness, which fails as $-\infty$. Consequently, $\langle -,- \rangle_{\mathcal{L}_{-\infty}}$ is an indefinite inner-product-type pairing on $(F^I,+_{-\infty,0},\cdot)$ (valued in $(F,+_{-\infty,0},\cdot)$). 

Next, we recover the \(\mathcal{L}^{-\infty}\) semi-norm. By Definition–Lemma~\ref{infty}, we have
\[
\|f\|_{\mathcal L^{-\infty}}
\;:=\;
\lim_{p \to -\infty}
\left(\int_{I} |f(x)|^{p}\,dx\right)^{\!1/p}
\;=\;
\operatorname*{ess\,inf}_{x\in I}\{|f(x)|\}.
\]

Accordingly, the corresponding space \(\mathcal L^{-\infty}(F)\) can be described as
\[
\mathcal L^{-\infty}(F)
\;=\;
\Bigl\{\, f\in \mathcal{F}_{\mathrm{int}}(I)\ \Big|\ 
\operatorname*{ess\,inf}_{x\in I}\{|f(x)|\}>-\infty \Bigr\}.
\]
The pairing \(\langle -,-\rangle_{\mathcal L^{-\infty}}\) endows the above structure with an indefinite
inner–product–type non-associative near–space over the non-associative near-line
\((F,+_{-\infty,0},\cdot)\), and its induced semi-norm is \(\|-\|_{\mathcal L^{-\infty}}\).

\section{Application: weighted generalized means for a family of complex numbers} 
\label{weighted}

Generalized means are fundamental to modern data science, physics and engineering,  particularly in regularization, loss functions, distance metrics, and sparse signal processing \cite{Brunton, James, Murphy}. 
Lemma~\ref{infty} provides a framework for defining maximum and minimum means, extending naturally to complex numbers—including non-positive real numbers. By employing the $\alpha$-sum operations, we can recover several classical means as particular cases. The use of $\alpha$-sums and their limits allows us to formulate a unified concept of \emph{generalized means} for arbitrary families of complex (or real) numbers. These constructions specialize to the familiar means when restricted to positive real values.

\medskip
Let \(\alpha \in \mathbb{R}^{+} \setminus \{0\}\), let \(\{\omega_k\}_{k \in [n]}\) be a family of positive real weights such that \(\sum_{k \in [n]} \omega_k = 1\), and let \(\{r_k\}_{k \in [n]}\) be a family of positive real numbers. Recall that the \emph{weighted power mean} of this family is defined by
\[
M_{\alpha}((r_{1},\omega_1), \ldots, (r_{n}, \omega_n))
\;=\;
\left(\sum_{j \in [n]} \omega_{j} r_{j}^{\alpha}\right)^{\frac{1}{\alpha}},
\]
where \(\alpha \in \mathbb{R}\) (cf.~\cite[pp.~175–265]{Bullen}).

\medskip
We now express this power mean in terms of the $+_{\alpha}$-sum and then show how this formulation extends naturally to complex numbers.

\smallskip
Write each weight as
\[
\omega_{k} \;=\; e^{i\theta_{k}} + e^{-i\theta_{k}},
\qquad
\theta_k = \operatorname{arccos}\!\left(\frac{\omega_k}{2}\right) \in \left(-\frac{\pi}{2}, \frac{\pi}{2}\right).
\]
Define
\[
a_{k,j} = r_{k}\,e^{i j \theta_{k}},
\qquad
k \in [n],\ j \in \{\pm 1\}.
\]
Then, for any exponent \(\alpha \in \mathbb{R}\), the usual weighted power mean can be rewritten as
\begin{align*}
 {}^{\alpha}\!\!\!\!\!\!\!\!\!\!\!\!\!\!\sum_{(j,k)\in [n]\times\{\pm 1\}} a_{k,j} 
&= \epsilon_{\alpha}^{-1}\!\left(\sum_{k \in [n]} \epsilon_{\alpha}\!\left(r_k e^{i\theta_k} + r_k e^{-i\theta_k}\right)\right)\\
&= \epsilon_{\alpha}^{-1}\!\left(\sum_{k \in [n]}\left(r_k^{\alpha} e^{i\theta_k} + r_k^{\alpha} e^{-i\theta_k}\right)\right)\\
&= \epsilon_{\alpha}^{-1}\!\left(\sum_{k \in [n]} r_k^{\alpha}(e^{i\theta_k}+e^{-i\theta_k})\right)\\
&= \epsilon_{\alpha}^{-1}\!\left(\sum_{k \in [n]} \omega_k r_k^{\alpha}\right)\\
&= \left(\sum_{k \in [n]} \omega_k r_k^{\alpha}\right)^{1/\alpha}\\
&= M_{\alpha}((r_{1}, \omega_1), \ldots, (r_{n}, \omega_n)).
\end{align*}

\medskip
Proceeding analogously with the sums defined through the limiting operations at
\(\pm\infty\), we obtain the corresponding extreme (maximum and minimum) means,
thereby extending the classical notion of power means to complex and non-associative
settings.

\begin{prop}
Let \(\{\omega_k\}_{k \in [n]}\) be a family of positive real numbers such that
\(\sum_{k \in [n]} \omega_k = 1\), and let \(\{r_k\}_{k \in [n]}\) be a family of positive
real numbers. For each \(k \in [n]\), write
\[
\omega_k = e^{i\theta_k} + e^{-i\theta_k},
\qquad
\theta_k = \arccos\!\left(\frac{\omega_k}{2}\right) \in \left(-\frac{\pi}{2}, \frac{\pi}{2}\right).
\]
Define \(a_{k,j} = r_k e^{i j \theta_k}\) for all \((k,j) \in [n] \times \{\pm 1\}\).
Then:
\begin{enumerate}
  \item \[
    {}^{\infty,0}\!\!\!\!\!\!\!\!\!\!\!\!\sum_{(k,j)\in [n]\times\{\pm 1\}} a_{k,j}
    \;=\;
    \max\{\,|r_k| \mid k\in [n]\,\};
  \]

  \item \[
    {}^{-\infty,0}\!\!\!\!\!\!\!\!\!\!\!\!\sum_{(k,j)\in [n]\times\{\pm 1\}} a_{k,j}
    \;=\;
    \min\{\,|r_k| \mid k\in [n]\,\};
  \]

  \item \[
    {}^{0,0}\!\!\!\!\!\!\!\!\!\!\!\!\sum_{(k,j)\in [n]\times\{\pm 1\}} a_{k,j}
    \;=\;
    \prod_{k\in [n]} r_k^{\omega_k}.
    \]
    This corresponds to the generalized weighted geometric mean with weights
    \(\omega_k\), for all \(k\in [n]\). In particular, for uniform weights
    \(\omega_k = \tfrac{1}{n}\), we recover the standard geometric mean of the
    \(r_k\)'s.

  \item For any \(\alpha \in \mathbb{R}\),
    \[
    {}^{\alpha}\!\!\!\!\!\!\!\!\!\!\!\!\!\!\sum_{(k,j)\in [n]\times\{\pm 1\}} a_{k,j}
    \;=\;
    \left(\sum_{k\in [n]} \omega_k r_k^{\alpha}\right)^{\!1/\alpha}.
    \]
    In particular, for uniform weights \(\omega_k = \tfrac{1}{n}\) and various
    values of \(\alpha\):
    \begin{itemize}
      \item \(\alpha > 0\): the power mean of the \(r_k\)'s,
      \item \(\alpha = -1\): the harmonic mean,
      \item \(\alpha = 1\): the arithmetic mean,
      \item \(\alpha = 2\): the quadratic (root–mean–square) mean,
      \item \(\alpha = 3\): the cubic mean.
    \end{itemize}
\end{enumerate}
\end{prop}

In light of these constructions, we now introduce the notion of \emph{weighted generalized means}, which unifies classical means and their extensions to complex and non-associative settings.

\begin{definlem}\label{weight}
Let $[n]$ be a finite set.  
Let $\{\omega_k\}_{k \in [n]}$ be a family of complex numbers, and let $\{r_k\}_{k \in [n]}$ be a family of complex numbers written as 
\[
r_k = |r_k| e^{i \phi_k},
\qquad 
k \in [n],
\]
satisfying the normalization condition
\[
\left|\sum_{k \in [n]} \omega_k e^{i \phi_k}\right| = 1.
\]
For each $k \in [n]$, express the weight $\omega_k$ as a finite sum
\[
\omega_k = \sum_{j \in [m_k]} e^{i \theta_{k,j}},
\qquad
\theta_{k,j} \in \mathbb{R},
\]
and define
\[
a_{k,j} = r_k e^{i \theta_{k,j}},
\qquad
(k,j) \in [n] \times [m_k].
\]
We define the \textsf{weighted generalized mean of type~$\alpha$} with respect to the weighted family 
\((r_1,\omega_1),\ldots,(r_n,\omega_n)\), denoted by
\[
M_{\alpha}((r_1,\omega_1),\ldots,(r_n,\omega_n))
\;:=\;
{}^{\alpha}\!\!\!\!\!\!\!\!\!\!\!\!\!\sum_{(k,j)\in [n]\times[m_k]} a_{k,j},
\]
where
\[
\alpha \in 
\mathbb{C}\setminus 
\Bigl(
i\mathbb{R}\;
\cup\;
\{(0,\Theta),(\infty,\Theta),(-\infty,\Theta)
  \mid 
  \Theta \in (-\tfrac{\pi}{2},\tfrac{\pi}{2}) \cup (\tfrac{\pi}{2},\tfrac{3\pi}{2})
\}
\Bigr).
\]
We distinguish the following cases:
\begin{itemize}
\item If \(\alpha \in \mathbb{C}\setminus i\mathbb{R}\), we call \(M_\alpha\) the \textsf{weighted power mean};
\item If \(\alpha = (0,\Theta)\), we call \(M_\alpha\) the \textsf{weighted geometric mean in the direction~$\Theta$};
\item If \(\alpha = (\infty,\Theta)\), we call \(M_\alpha\) the \textsf{weighted maximum in the direction~$\Theta$};
\item If \(\alpha = (-\infty,\Theta)\), we call \(M_\alpha\) the \textsf{weighted minimum in the direction~$\Theta$}.
\end{itemize}
\end{definlem}

\begin{proof}
We first show that any complex number \(\omega_k\) can be expressed as a sum of complex numbers of modulus~1.  
Indeed, writing \(\omega_k = |\omega_k| e^{i\theta_{\omega_k}}\), we note that since \(|\omega_k|\) is real, one can write
\[
|\omega_k| = e^{i\beta_k} + e^{-i\beta_k},
\qquad
\beta_k = \arccos\!\left(\frac{|\omega_k|}{2}\right) \in \left(-\frac{\pi}{2},\frac{\pi}{2}\right).
\]
Hence,
\[
\omega_k 
= |\omega_k| e^{i\theta_{\omega_k}}
= (e^{i\beta_k}+e^{-i\beta_k}) e^{i\theta_{\omega_k}}
= e^{i(\beta_k+\theta_{\omega_k})} + e^{-i(\beta_k-\theta_{\omega_k})},
\]
which expresses \(\omega_k\) as a sum of two complex numbers of unit modulus.

\medskip
Next, we show that the generalized weighted means are independent of the choice of \(\theta_{k,j}\)'s.  
Suppose that for a fixed \(k\in[n]\),
\[
\omega_k = \sum_{j \in [m_k]} e^{i\theta_{k,j}}
= \sum_{j \in [m_k']} e^{i\theta_{k,j}'},
\]
where \(\theta_{k,j}, \theta_{k,j}' \in \mathbb{R}\) and \(m_k,m_k' \in \mathbb{N}\).
Define the corresponding families
\[
a_{k,j} = r_k e^{i\theta_{k,j}}
\quad\text{and}\quad
a_{k,j}' = r_k e^{i\theta_{k,j}'},
\]
and index sets
\[
J_k =[n]\times [m_k],
\qquad
J_k' =[n]\times [m_k'].
\]
Then,
\[
\begin{aligned}
{}^{\alpha}\!\!\!\!\!\sum_{(k,j)\in J_k} a_{k,j}
&= \epsilon_{\alpha}^{-1}\!\left(\sum_{(k,j)\in J_k} |r_k|^{\alpha} e^{i(\phi_k+\theta_{k,j})}\right)
= \epsilon_{\alpha}^{-1}\!\left(\sum_{k\in[n]} |r_k|^{\alpha} e^{i\phi_k} \sum_{j\in[m_k]} e^{i\theta_{k,j}}\right)\\
&= \epsilon_{\alpha}^{-1}\!\left(\sum_{k\in[n]} |r_k|^{\alpha} e^{i\phi_k} \sum_{j\in[m_k']} e^{i\theta_{k,j}'}\right)
= \epsilon_{\alpha}^{-1}\!\left(\sum_{(k,j)\in J_k'} |r_k|^{\alpha} e^{i(\phi_k+\theta_{k,j}')}\right)\\
&= M_{\alpha}\big((a_{k,j}')_{(k,j)\in J_k'}\big).
\end{aligned}
\]
Thus, \(M_{\alpha}\) is independent of the particular choice of the angles \(\theta_{k,j}\).  
Since the other types of means (geometric, maximum, and minimum) arise as limits of \(M_{\alpha}\), this independence extends to all weighted generalized means defined above.
\end{proof}

This definition subsumes the classical power mean (for \(r_k>0\) and \(\alpha\in\mathbb{R}^+\setminus\{0\}\)) and extends seamlessly to datasets with negative real values (e.g., \(\phi_j=\pi\)) and to complex inputs. 

The construction is natural in algebraic frameworks isomorphic to \(\mathbb{R}\) or \(\mathbb{C}\) that retain the usual multiplication while deforming addition, so the familiar mean formulas transport unchanged to this setting.

Kolmogorov  (cf. \cite{Kolmogorov}) showed that any mapping $M: (\mathbb{R}^{+} \setminus \{0\})^{n} \mapsto \mathbb{R}^{+}\setminus \{0\}$ which obeys the 
following axioms is a generalised mean:
\begin{enumerate}
	\item[M1.] $M(r_{1}, \ldots, r_{n})$ is continuous and monotonic;
	\item[M2.] $M(r_{1}, \ldots, r_{n})$ is symmetric. In other words, is invariant under permutation of its
	arguments; 
	\item[M3.] the mean of repeated data equals the repeated value. In other words, $M(\underbrace{r,\ldots, r}_{\text{n times}}) = r$.  
	\item[M4.] the mean of the sample remains unchanged if a part of the sample is replaced by its corresponding
	mean. That is, if, say $m = M(r_{1}, \ldots, r_{k})$ where $k < n$, then $M(r_{1}, \ldots, r_{n}) = M(\underbrace{m,\ldots, m}_{\text{k times}}, r_{k+1}, \ldots, r_{n})$.
\end{enumerate}
Kolmogorov also proved that any mapping satisfying (M1)–(M4) has the representation
\[
M_{g}(r_{1},\ldots,r_{n})
=\;
g^{-1}\!\left(\frac{1}{n}\sum_{i\in[n]} g(r_{i})\right),
\]
for some continuous, monotone function \(g\).

In our setting, analogous means on \(\mathbb{R}^{+}\setminus\{0\}\) arise naturally by taking
\(g\) to be \emph{continuous multiplicative}; such \(g\) are in particular monotone and fit
the paradigm of keeping the usual multiplication while deforming addition in structures
isomorphic to \(\mathbb{R}\). Over \(\mathbb{C}\), “monotone” has no direct analogue, whereas
“multiplicative” does; axioms (M2)–(M4) remain meaningful in either case.

Kolmogorov’s theorem addresses equal weights; we adapt the argument to general weights and
verify that our means satisfy the weighted versions of (M2)–(M4). Property (M1) was discussed
earlier and is omitted here.
\begin{itemize}
    \item[\textbf{M2.}] This property holds because the addition operation is commutative.

    \item[\textbf{M3.}] For repeated values, where \(r_k = r\) for all \(k \in I\) and \(\omega_k = \omega = \frac{1}{n}\), we have
    \begin{align*}
        M_{\alpha}(\underbrace{(r,\omega),\ldots,(r,\omega)}_{\text{\(n\) times}})
        &= \epsilon_{\alpha}^{-1}\!\left(\sum_{k \in [n]} \omega\, \epsilon_{\alpha}(r)\right)
        = \epsilon_{\alpha}^{-1}\!\left(\frac{1}{n}\sum_{k \in [n]} \epsilon_{\alpha}(r)\right) \\
        &= \epsilon_{\alpha}^{-1}\!\left(\epsilon_{\alpha}(r)\right)
        = r.
    \end{align*}

    \item[\textbf{M4.}] We formulate the result by taking the first \(k\) points as a sub-sample, but the argument applies to any subset.  
    To compute the mean of a sub-sample, the weights must be normalized relative to that subset.  
    For the sub-sample of size \(k<n\), the adjusted weights are
    \[
    \omega_{j}' = \frac{\omega_{j}}{\sum_{i \in [k]} \omega_{i}},
    \quad\text{so that}\quad
    \sum_{j \in [k]} \omega_{j}' = 1.
    \]
    The mean of the sub-sample is then
    \[
    \overline{r} = M_{\alpha}((r_{1},\omega_{1}'),\ldots,(r_{k},\omega_{k}'))
    = \epsilon_{\alpha}^{-1}\!\left(\sum_{j \in [k]} \omega_{j}'\,\epsilon_{\alpha}(r_{j})\right).
    \]
    When each of the first \(k\) entries in the full sample is replaced by this mean value \(\overline{r}\), their new common weight is
    \[
    \overline{\omega} = \frac{1}{k}\sum_{j \in [k]}\omega_{j},
    \]
    ensuring that the total weight remains normalized:
    \(k\overline{\omega} + \sum_{j=k+1}^{n}\omega_{j} = 1\).

    The mean of the modified dataset is therefore
    \begin{align*}
        &M_{\alpha}(\underbrace{(\overline{r},\overline{\omega}),\ldots,(\overline{r},\overline{\omega})}_{\text{\(k\) times}},(r_{k+1},\omega_{k+1}),\ldots,(r_{n},\omega_{n})) \\
        &= \epsilon_{\alpha}^{-1}\!\left(
            k\overline{\omega}\,\epsilon_{\alpha}(\overline{r})
            + \sum_{j=k+1}^{n}\omega_{j}\,\epsilon_{\alpha}(r_{j})
        \right) \\
        &= \epsilon_{\alpha}^{-1}\!\left(
            k\!\left(\frac{1}{k}\sum_{i \in [k]}\omega_{i}\right)
            \epsilon_{\alpha}\!\left(
                \epsilon_{\alpha}^{-1}\!\left(
                    \sum_{j \in [k]}\omega_{j}'\,\epsilon_{\alpha}(r_{j})
                \right)
            \right)
            + \sum_{j=k+1}^{n}\omega_{j}\,\epsilon_{\alpha}(r_{j})
        \right) \\
        &= \epsilon_{\alpha}^{-1}\!\left(
            \sum_{i \in [k]}\omega_{i}
            \!\left(
                \sum_{j \in [k]}\frac{\omega_{j}}{\sum_{i \in [k]}\omega_{i}}
                \epsilon_{\alpha}(r_{j})
            \right)
            + \sum_{j=k+1}^{n}\omega_{j}\epsilon_{\alpha}(r_{j})
        \right) \\
        &= \epsilon_{\alpha}^{-1}\!\left(
            \sum_{j \in [n]}\omega_{j}\epsilon_{\alpha}(r_{j})
        \right)
        = M_{\alpha}((r_{1},\omega_{1}),\ldots,(r_{n},\omega_{n})).
    \end{align*}
    Hence, property (M4) is satisfied.
\end{itemize}

\end{document}